\documentclass[11pt]{amsart}
\usepackage{amsmath,amssymb}
\usepackage{graphicx}

\usepackage{fullpage}
\usepackage{enumerate}
\usepackage{pgf,tikz}
\usepackage{todonotes}
\usetikzlibrary{arrows}
\usetikzlibrary{positioning}
\usepackage{array}
\newcommand{\msc}[1]{\begin{center}MSC2000: #1.\end{center}}
\def\COMMENT#1{}
\def\TASK#1{}

\def\noproof{{\unskip\nobreak\hfill\penalty50\hskip2em\hbox{}\nobreak\hfill%
        $\square$\parfillskip=0pt\finalhyphendemerits=0\par}\goodbreak}
\def\endproof{\noproof\bigskip}
\newdimen\margin   % needed for macros \textdisplay & \ltextdisplay
\def\textno#1&#2\par{%
    \margin=\hsize
    \advance\margin by -4\parindent
           \setbox1=\hbox{\sl#1}%
    \ifdim\wd1 < \margin
       $$\box1\eqno#2$$%
    \else
       \bigbreak
       \hbox to \hsize{\indent$\vcenter{\advance\hsize by -3\parindent
       \sl\noindent#1}\hfil#2$}%
       \bigbreak
    \fi}
\def\proof{\removelastskip\penalty55\medskip\noindent{\bf Proof. }}

\newtheorem{thm}{Theorem}[section]
\newtheorem{define}[thm]{Definition}

\newtheorem{lem}[thm]{Lemma}
\newtheorem{claim}[thm]{Claim}
\newtheorem{fact}[thm]{Fact}

\newtheorem{conj}[thm]{Conjecture}
\newtheorem{prop}[thm]{Proposition}

\newtheorem{ques}[thm]{Question}
\newtheorem*{alg}{Algorithm}
\newtheorem*{thm*}{Theorem}
\newtheorem*{define*}{Definition}
\newtheorem*{examp*}{Example}
\newtheorem*{lem*}{Lemma}
\newtheorem*{claim*}{Claim}
\newtheorem*{fact*}{Fact}
\newtheorem*{col*}{Corollary}
\newtheorem*{conj*}{Conjecture}

\newtheorem*{tregsaythm*}{Theorem~\ref{tregsaythm}}

\newcommand{\rp}{\scriptsize\left(\!\!  
                \begin{array}{c} 
                r \\
                p
                \end{array} 
                \!\!\right)}

\begin{document}

\title{A degree sequence version of the K\"uhn--Osthus tiling theorem}
\date{}
\author{Joseph Hyde and Andrew Treglown}
\thanks{
JH: University of Birmingham, United Kingdom, {\tt jfh337@bham.ac.uk}. \\
\indent AT: University of Birmingham, United Kingdom, {\tt a.c.treglown@bham.ac.uk}.}

\begin{abstract}
A fundamental result of K\"uhn and Osthus [The minimum degree threshold for perfect
graph packings, Combinatorica,  2009] determines up to an additive constant the minimum degree threshold that forces a graph to contain a perfect $H$-tiling.
We prove a degree sequence version of this result which allows for a significant number of vertices to have lower degree.
\end{abstract}
\maketitle
\msc{05C35, 05C70}

\section{Introduction} \label{Introduction}
\subsection{Minimum degree conditions forcing tilings}
A substantial branch of extremal graph theory concerns the study of \emph{tilings}.
Given two graphs $H$  and $G$, an \emph{$H$-tiling} in $G$ 
is a collection of vertex-disjoint copies of $H$ in $G$. An
$H$-tiling is called \emph{perfect} if it covers all the vertices of $G$.
Perfect $H$-tilings are also often referred to as \emph{$H$-factors}, \emph{perfect $H$-packings} or \emph{perfect $H$-matchings}.

In the case when $H$ has a component on at least $3$ vertices, the decision problem of whether a graph contains a perfect $H$-tiling is NP-complete~\cite{hell}.
Thus, there has been a focus on establishing sufficient conditions to force a perfect $H$-tiling.
The seminal Hajnal--Szemer\'edi theorem~\cite{hs}
 characterises the minimum degree that ensures a graph contains a perfect $K_r$-tiling. 

\begin{thm}[Hajnal and Szemer\'edi~\cite{hs}]\label{hs}
Every graph $G$ whose order $n$
is divisible by $r$ and whose minimum degree satisfies $\delta (G) \geq (1-1/r)n$ contains a perfect $K_r$-tiling. Moreover, there are $n$-vertex graphs $G$
 with $\delta (G) = (1-1/r)n-1$ that do not contain a perfect $K_r$-tiling.
\end{thm}

The following result of Alon and Yuster~\cite{ay} shows that any sufficiently large graph $G$ with minimum degree slightly above that in Theorem~\ref{hs} in fact contains a perfect $H$-tiling for any graph $H$ with $\chi (H)=r$.
\begin{thm}[Alon and Yuster~\cite{ay}]\label{ay1}
Suppose that $\gamma >0$ and $H$ is a graph with $\chi (H)=r$. Then there exists an integer $n_0=n_0 (\gamma ,H)$
such that the following holds. If $G$ is a graph whose order $n \geq n_0$ is divisible by $|H|$ and
$$\delta (G) \geq (1-1/r+\gamma )n$$
then $G$ contains a perfect $H$-tiling.
\end{thm}
 For many graphs $H$ the minimum degree condition in Theorem~\ref{ay1} is best-possible up to the  term $\gamma n$.
Indeed, for many graphs $H$ there are so-called \emph{divisibility barrier} constructions $G$ on $n$ vertices that have minimum degree $(1-1/\chi (H))n-1$ but fail to contain a perfect $H$-tiling (see~\cite[Section 2]{kuhn2}).
 However, Koml\'{o}s, S\'ark\"{ozy} and  Szemer\'{e}di~\cite{kssAY} proved that the  term $\gamma n$ in Theorem~\ref{ay1} can be replaced
 with a constant dependent only on $H$.
Further, as discussed shortly, K\"uhn and Osthus~\cite{kuhn, kuhn2} proved that there are also many graphs $H$ for which one can significantly reduce the minimum degree condition in Theorem~\ref{ay1}.

In a related direction, Koml\'os~\cite{Komlos} showed that if one only requires an $H$-tiling covering \emph{almost} all vertices in the host graph, then one can replace the $\chi (H)$-term in the minimum degree condition of the Alon--Yuster theorem by the so-called \emph{critical chromatic number $\chi _{cr} (H)$ of $H$}. Here $\chi _{cr} (H)$  is defined
as 
\[\chi _{cr} (H):=(\chi (H)-1)\frac{|H|}{|H|-\sigma (H)},\]
where $\sigma (H)$ denotes the size of the smallest possible 
colour class in any $\chi(H)$-colouring of $H$. Note that all graphs $H$ satisfy $\chi (H) -1 < \chi _{cr} (H) \leq \chi (H)$ and $\chi _{cr} (H) = \chi (H)$ precisely when every $\chi (H)$-colouring $c$ of $H$ is balanced (i.e.
the colour classes of $c$ have the same size).
\begin{thm}[Koml\'os~\cite{Komlos}]\label{komcor}
Let $\eta > 0$ and let $H$ be a graph. Then there exists an integer $n_0=n_0(\eta, H)$ such that  every graph $G$ on $n \geq n_0$ vertices
 with
\[\delta (G) \geq \left ( 1 -\frac{1}{\chi _{cr} (H)} \right )n\]
 contains an $H$-tiling covering all but at most $\eta n$ vertices.
\end{thm}
Note that the minimum degree condition in Theorem~\ref{komcor} is best-possible in the sense that one cannot replace the   $(1 -{1}/{\chi _{cr} (H)})$ term here with any smaller fixed constant (this is a consequence of \cite[Theorem~7]{Komlos}).
Further, Koml\'os~\cite{Komlos} also determined the minimum degree threshold for ensuring a graph $G$ contains an $H$-tiling covering an $x$th proportion of its vertices, for any $x \in (0,1)$.
Shoukoufandeh and  Zhao~\cite{szhao} later proved that the number of uncovered vertices in Theorem~\ref{komcor} can be reduced to a constant dependent only on $H$.

 K\"uhn and Osthus~\cite{kuhn, kuhn2} showed that for many graphs $H$, a minimum degree slightly above that in Koml\'os' theorem actually ensures a \emph{perfect} $H$-tiling. To state their result we need to introduce some notation.
We say that a colouring of $H$ is \emph{optimal} if it uses exactly $\chi (H)=:r$
colours. Let $C_H$ denote the set of all optimal colourings of $H$.
Given an optimal colouring $c$ of $H$, let $x_{c,1} \leq x_{c,2} \leq  \dots \leq x_{c,r}$ denote
the sizes of the colour classes of~$c$. We write $$\mathcal D (c) := \{ x_{c,i+1} - x_{c,i}
\mid  i=1, \dots , r-1\},$$
and let 
$$\mathcal{D}(H) := \underset{c \in C_H}{\bigcup} \mathcal{D}(c).$$ 
We denote by ${\rm hcf}_{\chi} (H)$ the 
highest common factor of all integers in $\mathcal D(H)$. If $\mathcal D (H) =\{ 0 \}$
then we define ${\rm hcf} _{\chi} (H) := \infty $. We write ${\rm hcf}_c(H)$ for the
highest
common factor of all the orders of components of $H$. For non-bipartite graphs $H$
we say that
${\rm hcf} (H)=1$ if ${\rm hcf} _{\chi} (H)=1$. If $\chi (H)=2$ then we say
${\rm hcf} (H)=1$ if ${\rm hcf} _c (H)=1$ and $ {\rm hcf }_{\chi} (H)\leq 2$.
(See~\cite{kuhn2} for some examples.) Set
$$\chi ^* (H):=\begin{cases}
\chi _{cr} (H)& \text{if ${\rm hcf}(H)=1$;}\\
\chi (H)& \text{otherwise.}
\end{cases}$$
Also let $\delta (H,n)$ denote the smallest integer $k$ such that every graph $G$
whose order
$n$ is divisible by $|H|$ and with $\delta (G)\geq k$ contains a perfect $H$-tiling.

When $\text{hcf} (H)=1$, K\"uhn and Osthus showed that $\chi _{cr}(H)$ is the parameter governing the minimum degree condition that ensures a perfect $H$-tiling. When 
$\text{hcf} (H)\not =1$, $\chi (H)$ instead is the relevant parameter.
\begin{thm}[K\"uhn and Osthus~\cite{kuhn2}]\label{kothm1}
For every graph $H$ there exists a constant
$C=C(H)$ such that
$$\left(1-\frac{1}{\chi ^* (H)}\right)n-1 \leq \delta (H,n) \leq
\left(1-\frac{1}{\chi ^* (H)}\right)n+C.$$
\end{thm}
Intuitively speaking, graphs $H$ with  $\text{hcf} (H)=1$ avoid certain divisibility barrier problems when seeking a perfect $H$-tiling, thus ensuring the lower threshold in this case in Theorem~\ref{kothm1}.
Earlier K\"uhn and Osthus~\cite{kuhn} had proven a  version of Theorem~\ref{kothm1} for graphs $H$ with $\chi (H)\geq 3$ and $\text{hcf} (H)=1$; there though the constant $C(H)$ was replaced with a linear error term.
We now state this result explicitly for future reference.
\begin{thm}[K\"uhn and Osthus~\cite{kuhn}]\label{kothm}
Let $\eta > 0$ and $H$ be a graph with $\textnormal{hcf}_{\chi}(H) = 1$ and $\chi(H) =: r \geq 3$. Then there exists an integer $n_0 = n_0(\eta, H)$ such that the following holds. Let $G$ be a graph on $n \geq n_0$ vertices such that $|H|$ divides $n$ and
$$\delta(G) \geq \left(1 - \frac{1}{\chi_{cr}(H)} + \eta\right)n.$$ Then $G$ contains a perfect $H$-tiling.
\end{thm}

\subsection{Degree sequence conditions forcing tilings}\label{sectdegseqcond}
As discussed in the previous subsection, the minimum degree conditions in each of the Hajnal--Szemer\'edi theorem, Koml\'os' theorem and the K\"uhn--Osthus theorem are essentially best-possible. However,
this does not mean one cannot seek significant strengthenings of these results. For example,
 Kierstead and Kostochka~\cite{kier} proved an \emph{Ore-type} generalisation of Theorem~\ref{hs} where now one replaces the minimum degree condition with the condition that the sum of the degrees of every pair of non-adjacent vertices in $G$ is at least
$2(1-1/r)n-1$.

The focus of this paper concerns degree sequence conditions that force a perfect $H$-tiling. The study of degree sequence results for tilings was initiated in~\cite{bkt}.
In particular, a conjecture on a degree sequence strengthening of the Hajnal--Szemer\'edi theorem was raised~\cite[Conjecture 7]{bkt}, as well as a degree sequence version of the Alon--Yuster theorem~\cite[Conjecture 8]{bkt}.
In~\cite{Tregs} the second author proved the latter conjecture (also yielding an asymptotic version of Conjecture~7 from \cite{bkt}).
\begin{thm}[Treglown~\cite{Tregs}]\label{conj2} Suppose that $\eta >0$ and $H$ is a graph with $\chi (H)=:r\geq 2$. Then there exists an integer $n_0=n_0 (\eta ,H)$
such that the following holds. If $G$ is a graph whose order $n \geq n_0$ is divisible by $|H|$, and whose degree
sequence $d_1\leq \dots \leq d_n$ satisfies
 $$d_i \geq (r-2)n/r+i +\eta n  \  \text{ for all  } \ i < n/r,$$
then $G$ contains a perfect $H$-tiling. 
\end{thm}
Theorem~\ref{conj2} is a significant strengthening of the Alon--Yuster theorem as it allows for $n/r$ vertices to have degree (significantly) below that required in the latter.
Further Theorem~\ref{conj2} provides the first piece of a degree sequence analogue of the K\"uhn--Osthus theorem.

The main result in this paper deals with the remaining part of this problem, providing a degree sequence condition that forces a perfect $H$-tiling for graphs with $\text{hcf}(H)=1$.
\begin{thm}\label{mainthm}
Let $\eta > 0$ and $H$ be a graph with $\textnormal{hcf}(H) = 1$ and $\chi(H) =: r \geq 2$. Let $\sigma:= \sigma (H)$,  $h := |H|$ and $\omega := \left(h - \sigma\right)/(r-1)$.
Then there exists an integer $n_1 = n_1(\eta, H)$ such that the following holds. Let $G$ be a graph on $n \geq n_1$ vertices such that $h$ divides $n$ and $G$ has degree sequence $d_1 \leq \ldots \leq d_n$ such that

$$d_i \geq \left(1 - \frac{\omega+\sigma}{h}\right)n + \frac{\sigma}{\omega}i + \eta n\ \ \mbox{for all \ $1 \leq i \leq \frac{\omega n}{h}$.}$$
Then $G$ contains a perfect $H$-tiling. %

%\COMMENT{JH: I'm also pretty sure that the process should take around the same time as the original K\"{u}hn--Osthus theorem.}
\end{thm}
Observe that when $i = \omega n/h$, we have $$d_{\frac{\omega n}{h}} \geq \left(1 - \frac{\omega}{h} + \eta\right)n = \left(1 - \frac{1}{\chi_{cr}(H)} + \eta\right)n.$$ Hence, Theorem~\ref{mainthm} is a strengthening of Theorem~\ref{kothm}. Note that 
Theorem~\ref{mainthm} applies to all graphs $H$ with $\textnormal{hcf}(H) = 1$, not just graphs $H$ with $\chi(H) \geq 3$ and $\textnormal{hcf}_{\chi}(H) = 1$ (as in Theorem~\ref{kothm}).
Moreover, Theorem~\ref{mainthm} (and Theorem~\ref{conj2}) is best-possible for many graphs $H$ in the sense that we cannot replace the $\eta n$-term with a $o(\sqrt{n})$-term (see Proposition~\ref{square}). 
Theorem \ref{mainthm} is also best possible for all graphs $H$ in the sense that there are $n$-vertex graphs $G$ with only slightly 
 more than $\omega n/h$ vertices with degree (slightly) below $(1 - \omega/h + \eta)n$ that do not contain a perfect $H$-tiling (see Proposition~\ref{exex3prop}).
\COMMENT{AT:rewrote sentence as what was written wasn't correct} 
Thus, it is not possible to allow significantly more `small' degree vertices in Theorem~\ref{mainthm}.
Extremal examples are discussed in more detail in Section~\ref{sec:ex2}.

Combining Theorem~\ref{mainthm} with Theorem~\ref{conj2} we obtain the following degree sequence version of the K\"uhn--Osthus theorem (Theorem~\ref{kothm1}).

\begin{thm}\label{hydetregsdegseqkothm}
Let $\eta > 0$ and $H$ be a graph with $\chi(H) =: r \geq 2$. Let $\sigma:= \sigma (H)$,  $h := |H|$ and $\omega := \left(h - \sigma\right)/(r-1)$. Then there exists an integer $n_1 = n_1(\eta, H)$ such that if $G$ is a graph on $n \geq n_1$ vertices, $h$ divides $n$ and either \textnormal{(i)} or \textnormal{(ii)} below holds, then $G$ contains a perfect $H$-tiling.

\begin{itemize}
    \item[(i)] $\textnormal{hcf}(H) = 1$ and $G$ has degree sequence $d_1 \leq \ldots \leq d_n$ such that 

$$d_i \geq \left(1 - \frac{\omega+\sigma}{h}\right)n + \frac{\sigma}{\omega}i + \eta n\ \ \mbox{for all \ $1 \leq i \leq \frac{\omega n}{h}$.}$$
    \item[(ii)] $\textnormal{hcf}(H) \neq 1$ and $G$ has degree sequence $d_1 \leq \ldots \leq d_n$ such that

$$d_i \geq (r - 2)n/r + i + \eta n \ \ \mbox{for all} \ i < n/r.$$ 
\end{itemize}

\end{thm}

One can in fact obtain the following generalisation of Theorem~\ref{mainthm}.
\begin{thm}\label{mainthmsigma}
Let $\eta > 0$ and $H$ be a graph with $\textnormal{hcf}(H) = 1$ and $\chi(H) =: r \geq 2$. Let $h := |H|$. Set $\sigma \in \mathbb{R}$ such that $\sigma(H) \leq \sigma < h/r$ and $\omega := \left(h - \sigma\right)/(r-1)$.
Then there exists an integer $n_1 = n_1(\eta, H)$ such that the following holds. Let $G$ be a graph on $n \geq n_1$ vertices such that $h$ divides $n$ and $G$ has degree sequence $d_1 \leq \ldots \leq d_n$ such that 

$$d_i \geq \left(1 - \frac{\omega+\sigma}{h}\right)n + \frac{\sigma}{\omega}i + \eta n\ \ \mbox{for all \ $1 \leq i \leq \frac{\omega n}{h}$.}$$
Then $G$ contains a perfect $H$-tiling.\COMMENT{Deleted the footnote here as it wasn't quite correct and would be a bit confusing for the reader.}
\end{thm}
In Section~\ref{sectsetup}, we will prove Theorem \ref{mainthmsigma} directly. The proof of Theorem~\ref{mainthmsigma} follows that of Theorem~\ref{kothm} in~\cite{kuhn} closely. The main novelty of our proof is how we avoid divisibility barriers. For this we make use of an elementary number theoretic result for graphs with $\text{hcf}(H)=1$ (see Theorem~\ref{partitionthm}).
We also make  use of a recent degree sequence strengthening\COMMENT{AT: I think strengthening is fine as it does imply the original Komlos} of Koml\'os' theorem proved by the authors and Liu~\cite{hlt}. 
\COMMENT{AT: rewrote this paragraph}

Since the choice of $\sigma \in [\sigma (H), h/r)$ is arbitrary, note that Theorem~\ref{mainthmsigma} provides an infinite collection of degree sequences that force a perfect $H$-tiling. 
Having a higher value of $\sigma$ lowers the starting point of the degree sequence condition, but at the price of a steeper `slope' and higher value of $d_{\omega n/h}$ (see Figure~1).
\COMMENT{AT: added sentence}
As with Theorem~\ref{mainthm},
for many graphs $H$, each of these degree sequences is best-possible in the sense that we cannot replace the $\eta n$-term with a $o(\sqrt{n})$-term (see Section~\ref{sec:ex2}).
Note too that we cannot extend Theorem~\ref{mainthmsigma} to the case when $\sigma < \sigma (H)$. Indeed, in this case, if we set $\eta \ll 1$ then the degree sequence condition in Theorem~\ref{mainthmsigma} would allow
all vertices in $G$ to have degree below $(1-1/\chi_{cr} (H))n-1$; however, we know from Theorem~\ref{kothm1} that there are graphs $G$ that satisfy this condition and that do not contain perfect $H$-tilings.
\COMMENT{AT: added to explain why can't lower bound on $\sigma$.}

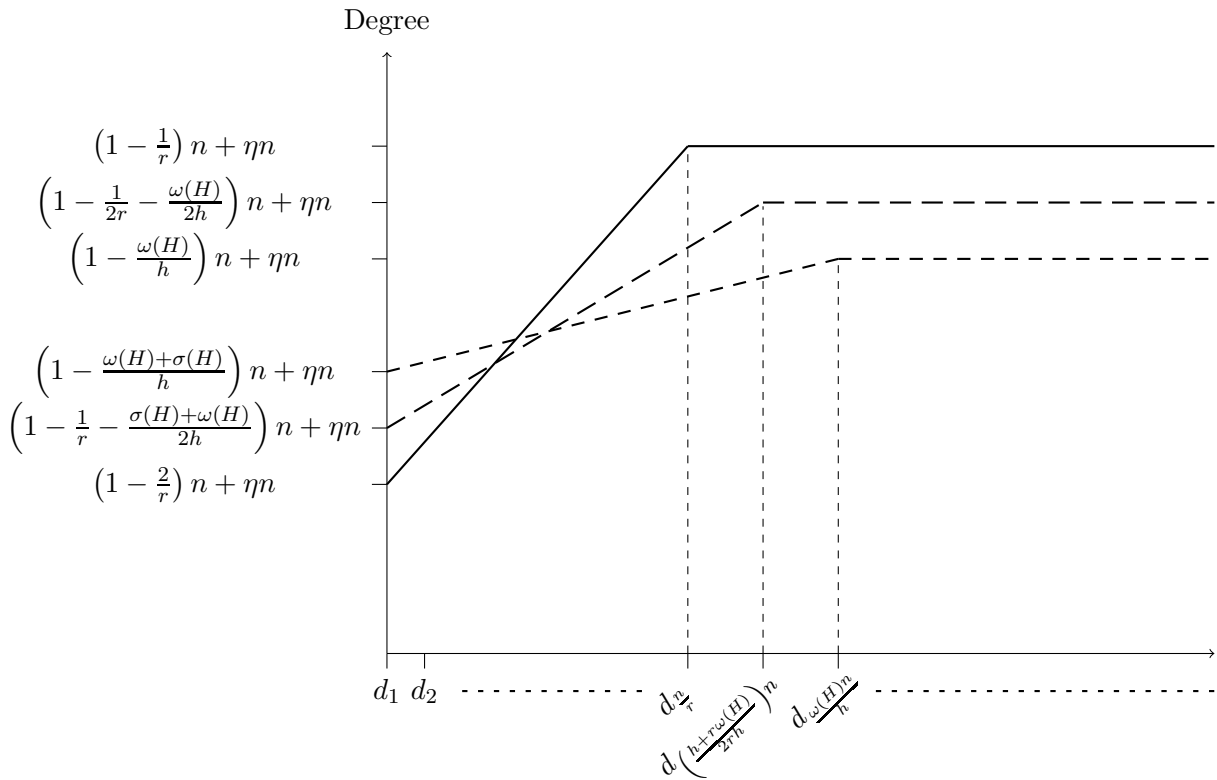
\begin{figure}
\begin{center}
\begin{tikzpicture}

\draw[->] (-4,0) -- (-4,8);
\draw[->] (-4,0) -- (7,0);
\draw (0,0) -- (0,-0.2);
\node at (-4,-0.5) {$d_1$};
\draw (-4,0) -- (-4,-0.2);
\draw (-3.5,0) -- (-3.5,-0.2);
\node at (-3.5,-0.5) {$d_2$};
\draw[dash pattern=on 2pt off 4pt, thick] (-3,-0.5) -- (-0.5,-0.5);
\draw[dash pattern=on 2pt off 4pt, thick] (2.5,-0.5) -- (7,-0.5);
\draw[black,thick] (-4,2.25) -- (0,6.75);
\draw[black,thick] (0,6.75) -- (7,6.75);
\draw[dash pattern=on 8pt off 4pt, thick] (-4,3) --  (1,6);
\draw[dash pattern=on 8pt off 4pt, thick] (1,6) -- (7,6);
\draw[dash pattern=on 5pt off 4pt, thick] (-4,3.75) --  (2,5.25);
\draw[dash pattern=on 5pt off 4pt, thick] (2,5.25) -- (7,5.25);
\node at (-4,8.4) {Degree};
\draw[dashed] (0,0) -- (0,6.70);
\draw[dashed] (1,0) -- (1,5.95);
\draw[dashed] (2,0) -- (2,5.20);
\draw (-4,2.25) -- (-4.2,2.25);
\draw (-4,3) -- (-4.2,3);
\draw (-4,3.75) -- (-4.2,3.75);
\draw (-4,5.25) -- (-4.2,5.25);
\draw (-4,6) -- (-4.2,6);
\draw (-4,6.75) -- (-4.2,6.75);
\node at (-6.7,2.25) {$\left(1-\frac{2}{r}\right)n + \eta n$};
\node at (-6.7,3) {$\left(1-\frac{1}{r} - \frac{\sigma(H) + \omega(H)}{2h}\right)n + \eta n$};
\node at (-6.7,3.75) {$\left(1 - \frac{\omega(H) + \sigma(H)}{h}\right)n + \eta n$};
\node at (-6.7,5.25) {$\left(1 - \frac{\omega(H)}{h}\right)n + \eta n$};
\node at (-6.7,6) {$\left(1-\frac{1}{2r} - \frac{\omega(H)}{2h}\right)n + \eta n$};
\node at (-6.7,6.75) {$\left(1-\frac{1}{r}\right)n + \eta n$};
\draw (1,0) -- (1,-0.2);
\draw (2,0) -- (2,-0.2);
\node at (-0.15,-0.65) {\rotatebox{45}{$d_{\frac{n}{r}}$}};
\node at (0.45,-1) {\rotatebox{45}{$d_{\left( \frac{h+r\omega(H)}{2rh}\right)n}$}};
\node at (1.85,-0.65) {\rotatebox{45}{$d_{\frac{\omega(H) n}{h}}$}};
\end{tikzpicture}

\caption{The degree sequence in Theorem \ref{mainthmsigma} for a fixed graph $H$ given $\sigma = \sigma(H)$ (medium dashed); $\sigma = \frac{h+r\sigma(H)}{2r}$ (long dashed); $\sigma = \frac{h}{r}$ (full). }

\end{center}
\end{figure}

\COMMENT{AT degree sequence in this theorem also best-possible in sense can't lower or increase range of $\sigma$? JH: We can't lower or increase the range. If we take $\sigma < \sigma(H)$ then we could use an extremal example akin Extremal Example 3 to show that we can't ensure a perfect $H$-tiling (I think). Also, $\sigma > h/r$ is nonsensical since then $\omega < \sigma$ and you can just switch the definitions of $\sigma$ and $\omega$ around.
AT: Joseph, your comment about the case when $\sigma > h/r$ isn't correct. The roles of $\sigma $ and $\omega$, are NOT interchangeable. PLEASE think more carefully about this question.}
\COMMENT{JH: Colour on Figure 1? AT no I
think it is ok to leave as it is}

\medskip

The paper is organised as follows. In the next section we discuss various senses of optimality for degree sequence conditions before giving several extremal examples for Theorems \ref{mainthm} and \ref{mainthmsigma}. We also ask whether one can improve Theorem \ref{mainthm} by suitably `capping' the bounds on the degrees of the vertices (see Question \ref{ques1}). In Section~\ref{sectnotate} we introduce some notation and definitions; in Section~\ref{sectaux} we give a number of auxiliary results and definitions
relating to the regularity lemma and tilings. We then prove an elementary number theoretic result (Theorem \ref{partitionthm}) in Section~\ref{sectstandalone} which will be a crucial tool in overcoming divisibility barriers during the proof of Theorem~\ref{mainthmsigma}. In Section~\ref{sectoverview} we give an overview of the proof of Theorem~\ref{mainthmsigma}
before proving it in 
 Section~\ref{sectsetup}.

\section{A discussion on the optimality of degree sequence conditions}\label{sec:ex}
In this section we describe various notions concerning when a degree sequence condition is `best-possible' in some sense. In particular, we will explain in what way our results (Theorems~\ref{mainthm} and~\ref{mainthmsigma}) are essentially best-possible, as well as how we may be able to strengthen these theorems further. Some of our discussion draws on the survey~\cite{survey}.

First we introduce a few definitions. An integer sequence $\pi =(d_1\leq \dots \leq d_n)$ is called \emph{graphical} if there exists a (simple) graph $G$ that has $\pi$ as its degree sequence.
Given a graph property $P$, we say that a graphic sequence $\pi$ \emph{forces} $P$ if every graph with degree sequence $\pi$ satisfies property $P$. Given a property $P$ (such as containing a Hamilton cycle or perfect $H$-tiling), the `holy-grail' in the study of degree sequences is to establish all
those graphic sequences $\pi$ that force $P$.

The following theorem of Chv\'atal~\cite{chv} provides an extremely general condition on degree sequences that force a Hamilton cycle.
\begin{thm}[Chv\'atal~\cite{chv}]\label{chthm}
Suppose that the degree sequence of a graph $G$
is $d_1\le \dots \le d_n$. If $n \geq 3$ and $d_i \geq i+1$ or $d_{n-i} \geq n-i$
for all $i <n/2$ then $G$ is Hamiltonian. 
\end{thm}
Note that 
Chv\'atal's theorem is best-possible in the following sense:  if $d_1 \leq \dots \leq d_n$
is a  degree sequence that does not satisfy the condition in Theorem~\ref{chthm} then there
exists a non-Hamiltonian graph $G$ whose degree sequence $d'_1 \leq  \dots \leq d'_n$ is such
that $d'_i \geq d_i$ for all $1\leq i \leq n$. (We will informally refer to a degree sequence result being best-possible in this way as a \emph{Chv\'atal-type} result.)

Crucially note though that Chv\'atal's theorem does not describe all those graphic sequences  that force a Hamilton cycle. For example,  graphs with degree sequence $(2,2,2,2,2)$ must be Hamiltonian (in fact, are themselves simply a $5$-cycle), but do not satisfy
Chv\'atal's condition. More generally,   all $2k$-regular graphs on $4k+1$ vertices are Hamiltonian~\cite{nw} yet their degree sequences fail the condition in Theorem~\ref{chthm}.

\subsection{Degree sequence conditions forcing perfect $H$-tilings}
At present, for a given fixed graph $H$, it seems out of reach to characterise those graphical degree sequences that force a perfect $H$-tiling, or obtain a Chv\'atal-type result in this setting.
Thus, it is natural to seek degree sequence conditions that force a perfect $H$-tiling, and are best-possible in some weaker sense. For example, consider the following conjecture:
\begin{conj}[Balogh, Kostochka and Treglown~\cite{bkt}] \label{bktconj}
Let $n$, $r \in \mathbb{N}$ such that $r$ divides $n$. Suppose that $G$ is a graph on $n$ vertices with degree sequence $d_1 \leq \ldots \leq d_n$ such that:
\begin{itemize}
     \item[($\alpha$)] $d_i \geq (r - 2)n/r + i$ {for all} $ i < n/r$; 
      \item[($\beta$)] $  d_{n/r+1} \geq (r - 1)n/r$.
\end{itemize} Then $G$ contains a perfect $K_r$-tiling.
\end{conj} 
Conjecture~\ref{bktconj} is best-possible in the sense that there are examples (see  \cite[Section 4]{bkt}) showing that one cannot replace ($\alpha$) with $d_i \geq (r - 2)n/r + i - 1$ for even a \emph{single} $i$ or ($\beta$) with
 $d_{n/r+1} \geq (r - 1)n/r - 1$ and $d_{n/r+2} \geq (r - 1)n/r $. That is, there is no room to lower the degree sequence condition further, not even by lowering a single entry by just one. 
(We will informally refer to a degree sequence result being best-possible in this way as a \emph{P\'osa-type} result.)
However, Conjecture~\ref{bktconj}, if true, is likely still  significantly weaker than a Chv\'atal-type result. For example, it is easy to see that any graph $G$ with degree sequence $d_1\leq \dots \leq d_n$ satisfying (i) $d_1\geq r-1$;
(ii) $d_2 \geq n-2$; (iii) $d_{n-r+2}\geq n-1$ contains a perfect $K_r$-tiling even though the condition in Conjecture~\ref{bktconj} is violated.

One could also ask for a P\'osa-type strengthening of the K\"uhn--Osthus theorem (Theorem~\ref{kothm1}) for all graphs $H$.
However, obtaining such a result again seems extremely difficult, not only because (the special case) Conjecture~\ref{bktconj} is still open, but also in general because the `correct' value of the constant $C(H)$ in Theorem~\ref{kothm1} is not known.

\subsection{Extremal examples for Theorems~\ref{mainthm} and~\ref{mainthmsigma}}\label{sec:ex2}
Despite the aforementioned challenges, in this paper we have provided degree sequence conditions that force a perfect $H$-tiling, and are best-possible in various ways. The following 3 extremal examples demonstrate this. The first shows that we cannot significantly lower \emph{every} term in the degree sequence conditions of Theorems~\ref{mainthm} and~\ref{mainthmsigma} and still ensure a perfect $H$-tiling
for complete $r$-partite graphs $H$.
 The second shows that that the `slope' of the degree sequence in Theorem~\ref{mainthm} is best possible for so-called bottle graphs. The third demonstrates that for any graph $H$, to ensure a perfect $H$-tiling (or even an `almost' perfect $H$-tiling) in a graph $G$ on $n$ vertices we cannot have significantly more than $\omega n/h$ vertices that have degree below $\left(1 - \frac{1}{\chi_{cr}(H)} + \eta\right)n$.

\smallskip

\noindent
\emph{\bf{Extremal Example 1.}}
\textnormal{The following construction (a simple adaption of~\cite[Proposition 3.1]{Tregs}) demonstrates that for most complete $r$-partite graphs $H$, one cannot replace the $\eta n$-term in Theorems~\ref{mainthm} and~\ref{mainthmsigma}
with a $o(\sqrt{n})$-term.}
\begin{prop}\label{square} 
Let $r \geq 3$ and $H:=K_{t_1, \dots , t_r}$ with $t_i \geq 2$ (for all $1 \leq i \leq r$). Let $h := |H|$. Set $\sigma \in \mathbb{R}$ such that $\sigma(H) \leq \sigma < h/r$ and $\omega := \left(h - \sigma\right)/(r-1)$.
Let $n \in \mathbb N$ be sufficiently large so that $\sqrt{n}$ is an integer that is divisible by $6h^2$. Set $C:=\sqrt{n}/3h^2$.
Then there exists a graph $G$ on $n$ vertices whose degree sequence $d_1 \leq \dots \leq d_n$ satisfies
$$d_i \geq \left(1 - \frac{\omega+\sigma}{h}\right)n + \frac{\sigma}{\omega}i + C\ \ \mbox{for all \ $1 \leq i \leq \frac{\omega n}{h}$}$$
but such that $G$ does not contain a perfect $H$-tiling.
\end{prop}
\proof 
\textnormal{Let $G$ denote the graph on $n$ vertices consisting of $r$ vertex classes $V_1, \dots , V_r$ with $|V_1|=1$, $|V_2|= \omega n/h+1+Cr$, $|V_3|= (\sigma +\omega)n/h -2 - 3 C$ and
$|V_i|=\omega n/h- C$ if $4 \leq i \leq r$ and which contains the following edges:}
\begin{itemize}
\item \textnormal{All possible edges with an endpoint in $V_3$ and the other endpoint in $V(G)\setminus V_1$. (In particular, $G[V_3]$ is complete.);}
\item \textnormal{All edges with an endpoint in $V_2$ and the other endpoint in $V(G)\setminus V_2$;}
\item \textnormal{All edges with an endpoint in $V_i$ and the other endpoint in $V(G)\setminus V_i$ for $4 \leq i \leq r$;}
\item \textnormal{There are $\sqrt{n}/2$ vertex-disjoint stars in $V_2$, each of size $\lfloor 2|V_2|/\sqrt{n} \rfloor$,
$\lceil 2|V_2|/\sqrt{n} \rceil$, which cover all of $V_2$.}
\end{itemize}
%\begin{figure}\label{picture1}
%\begin{center}\footnotesize
%%\includegraphics[width=0.5\columnwidth]{_corradi2.eps}
%\caption{An example of a graph $G$ from Proposition~\ref{square} in the case when $r=3$.}
%\end{center}
%\end{figure}
\textnormal{In particular, note that the vertex $v \in V_1$ sends all possible edges to $V(G) \setminus V_3$ but no edges to $V_3$.}

\textnormal{Let $d_1 \leq \dots \leq d_n$ denote the degree sequence of $G$.
Notice that every vertex in $V_i$ for $3 \leq i \leq r$ has degree at least $(1-\omega/h)n+C$. 
Note that $\lfloor 2|V_2|/\sqrt{n} \rfloor \geq 2 \sqrt{n}/h =6Ch \geq 6Cr$.
Thus, there are $\sqrt{n}/2$ vertices in $V_2$ of degree at least
$$(1-\omega /h)n-1-Cr +(6Cr-1)  \geq (1-\omega /h)n+C.$$
The remaining $\omega n/h+1+Cr -\sqrt{n}/2 \leq \omega n/h - \sqrt{n}/3 -1$ vertices in $V_2$ have degree at least}
$$(1-\omega /h)n-Cr\geq (1-\omega /h)n  -\sigma \sqrt{n}/3\omega + C.$$
\textnormal{Since $d_G (v)\geq \left(1 - \frac{\omega+\sigma}{h}\right)n  + C+\sigma/\omega$ for the vertex $v \in V_1$ we have that}
$$d_i \geq \left(1 - \frac{\omega+\sigma}{h}\right)n + \frac{\sigma}{\omega}i + C\ \ \mbox{for all \ $1 \leq i \leq \frac{\omega n}{h}$}.$$

\textnormal{Suppose that $v \in V_1$ lies in a copy $H'$ of $H$ in $G$.
Then by construction of $G$, two of the vertex classes $U_1,U_2$ of $H'$ must lie entirely in $V_2$. By definition of $H$, $H'[U_1\cup U_2]$ contains a path of length $3$. However, $G[V_2]$ does not contain a path of
length $3$, a contradiction. Thus, $v$ does not lie in a copy of $H$ and so $G$ does not contain a perfect $H$-tiling.}
\endproof

\smallskip

\noindent
\emph{\bf{Extremal Example 2.}}
We require the following definitions. Let $t \in \mathbb{N}$. We will refer to a vertex class of size $t$ of $G$ as a \textit{$t$-class} of $G$. Set $r, \sigma, \omega \in \mathbb{N}$ and $\sigma <\omega$. We define the \textit{$r$-partite bottle graph $B$ with neck $\sigma$ and width $\omega$} to be the complete $r$-partite graph with one $\sigma$-class and $(r-1)$ $\omega$-classes.

 Let $\eta > 0$ be fixed. Let $B$ be an $r$-partite bottle graph with neck $\sigma$ and width $\omega$. The following extremal example (adapted from Proposition 3.1 in \cite{hlt}) $G$ on $n$ vertices demonstrates that Theorem~\ref{mainthm} is best possible for such graphs $B$, in the sense that $G$ satisfies the degree sequence of Theorem~\ref{mainthm} except for a small linear part that only just fails the degree sequence, but does not contain a perfect $B$-tiling. In fact, $G$ does not contain a $B$-tiling that covers all but at most $\eta n$ vertices.

\begin{prop}\label{exex2}
Let $\eta > 0$ be fixed and $n \in \mathbb{N}$ such that $0 < 1/n \ll \eta \ll 1$. Let $B$ be a bottle graph with neck $\sigma$ and width $\omega$, where $b := |B|$. 
Additionally assume that $b$ divides $n$.
Then for any $1 \leq k < \omega n/b - (rb + 1)\eta n$, there exists a graph $G$ on $n$ vertices whose degree sequence $d_1 \leq \ldots \leq d_n$ satisfies 
\[d_i \geq \left(1 - \frac{\omega+\sigma}{b}\right)n + \frac{\sigma}{\omega}i + \eta n \ \ \mbox{for all \ $i \in \{1, \ldots, k-1, k+rb\eta n +1, \ldots, \omega n/b$\},}\] 
\[d_i = \left(1 - \frac{\omega+\sigma}{b}\right)n + \left\lceil\frac{\sigma}{\omega}k\right\rceil + \eta n \ \ \mbox{for all \ $k \leq i \leq k+rb\eta n$,}\]  but such that $G$ does not contain a $B$-tiling covering all but at most $\eta n$ vertices.
\end{prop}

\begin{proof} \textnormal{Let $G$ be the graph on $n$ vertices with $r+1$ vertex classes $V_1, \ldots, V_{r+1}$ where }
\begin{itemize}
    \item $|V_1| = \sigma n/b$;
    \item $|V_2| = \omega n/b - \eta n$; 
    \item $|V_3| = \ldots = |V_r| = \omega n/b - (\eta n + 1)$;
    \item $|V_{r+1}| = (r-1)(\eta n + 1) - 1$.
\end{itemize} \textnormal{Label the vertices of $V_1$ as $a_1, a_2, \ldots, a_{\sigma n/b}$. Similarly, label the vertices of $V_2$ as $c_1, c_2, \ldots, c_{\omega n/b - \eta n}$. The edge set of $G$ is constructed as follows.}

\textnormal{Firstly, let $G$ have the following edges:}

\begin{itemize}
     \item \textnormal{All edges with an endpoint in $V_1$ and the other endpoint in $V(G)\setminus V_2$, in particular $G[V_1]$ is complete;}
     \item \textnormal{All edges with an endpoint in $V_i$ and the other endpoint in $V(G)\setminus (V_1 \cup V_i) $ for $2  \leq i \leq r+1$;}
     \item \textnormal{All edges with both endpoints in $V_{r+1}$, in particular $G[V_{r+1}]$ is complete;}
     \item \textnormal{Given any $1 \leq i \leq \omega n/b - \eta n$ and $j \leq \lceil\sigma i/\omega\rceil$ include all edges  $c_i a_j$.}
\end{itemize}

\textnormal{So at the moment $G$ does satisfy the degree sequence in Theorem~\ref{mainthm}; we therefore modify $G$ slightly.
For all $k~\leq~i~\leq~k~+~rb\eta n$ and $\lceil\sigma k/\omega \rceil + 1\leq j \leq \lceil\sigma (k+rb\eta n)/\omega\rceil$ delete each edge between $c_i$ and $a_j$.
One can easily check that $G$ satisfies the degree sequence in the statement of the proposition. In particular, the vertices of degree $\left(1 - \frac{\omega+\sigma}{b}\right)n + \lceil\frac{\sigma}{\omega}k\rceil + \eta n$ are $c_{k},\dots,
c_{k+rb\eta n}$.}

\textnormal{Define $A := \{a_1, \ldots, a_{\lceil\sigma k/\omega\rceil}\}$ and $C := \{c_1, \ldots, c_{k + rb\eta n}\}$. Note that there are no edges between $C$ and $V_1 \setminus A$ in $G$.}

\begin{claim} \label{exexampclaim} \textnormal{Let $T$ be a $B$-tiling of $G$. Then $T$ does not cover at least $3\eta n/2$ vertices in $C$.}
\end{claim}
\textnormal{Firstly, consider any copy $B'$ of $B$ in $T$ that contains at least one vertex in $V_{r+1}$. Since $C$ is an independent set in $G$, observe that $B'$ contains at most $\omega$ vertices from $C$. Thus there are at most $\omega|V_{r+1}| = \omega(r-1)\eta n + \omega(r-2)$ vertices in $C$ covered by copies of $B$ in $T$ that each contain at least one vertex in $V_{r+1}$.} 

\textnormal{Secondly, consider any copy $B'$ of $B$ in $T$ that contains at least one vertex from $C$ and no vertices from $V_{r+1}$. 
As before, since $C$ is an independent set in $G$, we have that $B'$ contains at most $\omega$ vertices from $C$.
Since there are no edges between $C$ and $V_1 \setminus A$ in $G$,
 $B'$ contains at least $\sigma$ vertices in $A$.} 

\textnormal{These two observations imply that at most $\omega(r-1)\eta n + \omega(r-2) + \lceil\sigma k/\omega\rceil(\omega/\sigma)< k + \left(b(r-1) + \frac{1}{2}\right)\eta n$ vertices in $C$ can be covered by $T$. Since $|C| = k + rb\eta n$, we have that $T$ does not cover at least $3\eta n/2$ vertices in $C$. Therefore, Claim \ref{exexampclaim} holds. Hence $G$ does not have a $B$-tiling covering all but at most $\eta n$ vertices.}
\end{proof}

\smallskip

\noindent
\emph{\bf{Extremal Example 3.}} 
Let $H$ be an $h$-vertex graph, $\chi (H)=:r$, $\sigma:=\sigma (H)$ and $\omega :=(h-\sigma)/(r-1)$.
The following extremal example  demonstrates that there are $n$-vertex graphs $G$  for which all but slightly more than $\omega n/h$ vertices have degree above 
$(1-1/\chi _{cr} (H)+o(1))n$, and the remaining vertices have degree $(1-1/\chi _{cr} (H)-o(1))n$, and yet $G$ does not contain a perfect $H$-tiling.
Thus, this shows that one cannot have significantly more than $\omega n/h$ `small' degree vertices in Theorem~\ref{mainthm}.
\COMMENT{AT: rewritten this paragraph as what was written was very wrong. Please check why. In particular, you could have more than $\omega n/h$ small degree vertices and still contain a perfect $H$-tiling (e.g. Let $G$ be a collection of disjoint $H$)}
\COMMENT{AT: replaced proposition with a much easier version. Please go through}
\begin{prop}\label{exex3prop}
Let $\eta > 0$ be fixed. Let $H$ be a graph with $\chi(H) =: r$. Let $h:=|H|$, $\sigma := \sigma(H)$ and set $\omega := (h-\sigma)/(r-1)$. Then there exists a graph $G$ on $n$ vertices whose degree sequence $d_1 \leq \ldots \leq d_n$ satisfies 
$$d_i = \left(1 - \omega/h - (r-1)\eta\right)n =\left (1-1/\chi _{cr}(H) -(r-1)\eta \right )n\ \ \text{for all } \ i \leq (\omega/h  + (r-1)\eta)n,$$
$$d_i \geq \left(1 - \omega/h + \eta\right)n= \left (1-1/\chi _{cr}(H) +\eta \right )n \ \ \text{ for all  } \ i > (\omega/h  + (r-1)\eta) n,$$
but such that $G$ does not contain an $H$-tiling covering all but at most $\eta n$ vertices.
\end{prop}
\begin{proof} Let $G$ be the complete $r$-partite graph on $n$ vertices with vertex classes $V_1, \ldots, V_{r}$ such that
\begin{itemize}
\item \textnormal{$|V_1| = \frac{\sigma n}{h} - \eta n$, }
\item \textnormal{$|V_2| = \frac{\omega n}{h} + (r-1)\eta n$, }
\item \textnormal{$|V_3| = \ldots = |V_r| = \frac{\omega n}{h} - \eta n $.}
\end{itemize} 
Then $G$ satisfies the degree sequence condition in the proposition. The choice in size of $V_1$ ensures that any $H$-tiling in $G$ covers at most $|V_1|h/\sigma< n- \eta n$ vertices, as desired.
\end{proof}

\subsection{A possible strengthening of Theorem~\ref{mainthm}}
Whilst Proposition~\ref{square} demonstrates that we cannot lower \emph{every} term in the degree sequence condition in Theorem~\ref{mainthm} by much, perhaps one can cap the degrees as follows.
\begin{ques}\label{ques1}
Can the degree sequence condition in Theorem~\ref{mainthm} be replaced by 
$$d_i \geq \min \left \{ \left(1 - \frac{\omega+\sigma}{h}\right)n + \frac{\sigma}{\omega}i + \eta n, 
\left(1 - \frac{1}{\chi _{cr} (H)}\right)n+C \right \}
\ \ \mbox{for all \ $1 \leq i \leq \frac{\omega n}{h}$}$$
where $C$ is a constant dependent only on $H$?
\end{ques}

Note that Theorem~\ref{hydetregsdegseqkothm} does not quite imply the K\"uhn--Osthus theorem (Theorem~\ref{kothm1}) due to the $\eta n$-terms. On the other hand, an affirmative answer to Question~\ref{ques1}, together with 
an analogous `capped' version of Theorem~\ref{hydetregsdegseqkothm}(ii), would fully imply the upper bound in Theorem~\ref{kothm1}.
\COMMENT{JH: Just to check, would we also need to show that $C\geq -1$ and for some graphs $H$ that $C(H) = -1$?
AT: No, I think you must have miss understood something. The lower bound just comes from a construction of a graph with min degree $(1-1/\chi _{cr} (H))n-1$ without a perfect $H$-tiling}

%%%%%%%%
\section{Notation and Definitions}\label{sectnotate}

Let $G$ be a graph. We define $V(G)$ to be the vertex set of $G$ and $E(G)$ to be the edge set of $G$. Let $X \subseteq V(G)$. Then $G[X]$ is the \textit{graph induced by $X$ on $G$} and has vertex set $X$ and edge set $E(G[X]) := \{xy \in E(G): x,y \in X\}$. We also define $G\setminus X$ to be the graph with vertex set $V(G)\setminus X$ and edge set $E(G\setminus X):= \{xy \in E(G): x,y \in V(G)\setminus X\}$. For each $x \in V(G)$, we define the \textit{neighbourhood of $x$ in $G$} to be $N_G(x):= \{y\in V(G): xy \in E(G)\}$ and define $d_G(x) := |N_G(x)|$. We drop the subscript $G$ if it is clear from context which graph we are considering. 
We write $d_G(x,X)$ for the number of edges in $G$ that $x$ sends to vertices in $X$. Given a subgraph $G'\subseteq G$, we will write $d_G(x,G'):=d_G(x,V(G'))$.
Let $A, B \subseteq V(G)$ be disjoint. Then we define $e_G(A,B):=|\{xy \in E(G): x\in A, y\in B\}|$.

Let $t \in \mathbb{N}$. We define the \emph{blow-up} $G(t)$ to be the graph constructed by first replacing each vertex $x \in V(G)$ by a set $V_x$ of $t$ vertices and then replacing each edge $xy \in E(G)$ with the edges of the complete bipartite graph with vertex sets $V_x$ and $V_y$.

We write $0 < a \ll b \ll c < 1$ to mean that we can choose the constants $a,b,c$ from right to left. More precisely, there exist non-decreasing functions $f: (0,1] \to (0,1]$ and $g: (0,1] \to (0,1]$ such that for all $a \leq f(b)$ and $b \leq g(c)$ our calculations and arguments in our proofs are correct. Larger hierarchies are defined similarly. Note that $a \ll b$ implies that we may assume e.g. $a < b$ or $a < b^2$.

\section{Auxiliary results}\label{sectaux}
\subsection{The regularity and blow-up lemmas}
The results in this section will be employed in our proof of Theorem~\ref{mainthmsigma}. First we need the following definitions.

\begin{define}
\textnormal{Let $G =(A,B)$ be a bipartite graph with vertex classes $A$ and $B$. We define the \textit{density} of $G$ to be $$d_G(A,B) := \frac{e_G(A,B)}{|A||B|}.$$ 
Set $\varepsilon > 0$. We say that $G$ is \textit{$\varepsilon$-regular} if for all $X \subseteq A$ and $Y \subseteq B$ with $|X| > \varepsilon|A|$ and $|Y| > \varepsilon|B|$ we have that $|d_G(X,Y) - d_G(A,B)| < \varepsilon$.} 
\end{define}

\begin{define}
\textnormal{Given $\varepsilon > 0$, $d \in [0, 1]$ and $G = (A,B)$ a bipartite graph, we say that $G$ is \emph{$(\varepsilon, d)$-superregular} if all sets $X \subseteq A$ and $Y \subseteq B$ with $|X| \geq \varepsilon|A|$ 
and $|Y| \geq \varepsilon|B|$ satisfy that $d(X, Y ) > d$, that $d_G(a) > d|B|$ for all $a \in A$ and that $d_G(b) > d|A|$ for all $b \in B$.}
\end{define}

The following groundbreaking result of Szemer\'{e}di \cite{sze} will be instrumental in our proof of Theorem \ref{mainthmsigma}.

\begin{lem}[Degree form of Szemer\'{e}di's Regularity lemma \cite{sze}] \label{degformreglemma} 
Let $\varepsilon \in (0,1)$ and $M' \in \mathbb{N}$. Then there exist natural numbers $M$ and $n_0$ such that for any graph $G$ on $n \geq n_0$ 
vertices and any $d \in (0,1)$ there is a partition of the vertices of $G$ into subsets $V_0, V_1, \ldots, V_k$ and a spanning subgraph $G'$ of $G$ such that the following hold:

\begin{itemize}
    \item $M' \leq k \leq M$;
    \item $|V_0| \leq \varepsilon n$;
    \item $|V_1| = \ldots = |V_k| =: q$;
    \item $d_{G'}(x) > d_{G}(x) - (d+ \varepsilon)n$ for all $x \in V(G)$;
    \item $e(G'[V_i]) = 0$ for all $i \geq 1$;
    \item For all $1 \leq i,j \leq k$ with $i \neq j$ the pair $(V_i, V_j)_{G'}$ is $\varepsilon$-regular and has density either $0$ or at least $d$.
\end{itemize}
\end{lem}

We call $V_1, \ldots, V_{k}$ the \textit{clusters} of our partition, $V_0$ the exceptional set and $G'$ the \textit{pure graph}. We define the \textit{reduced graph} $R$ of $G$ with parameters $\varepsilon$, $d$ and $M'$ to be the graph whose vertex set is $V_1, \ldots, V_k$ and in which $V_iV_j$ is an edge if and only if $(V_i, V_j)_{G'}$ is $\varepsilon$-regular with density at least $d$. Note also that $|R| = k$.

We will apply the following well known fact, in conjunction with Lemma \ref{keylem} (below), in Section~\ref{sectabsorb}.

\begin{fact} \label{slicinglemma}
Let $0 < \varepsilon < \alpha$ and $\varepsilon':= \max\{\varepsilon/\alpha, 2\varepsilon\}$. Let $(A,B)$ be an $\varepsilon$-regular pair of density $d$.
Suppose $A' \subseteq A$ and $B' \subseteq B$ where $|A'| \geq \alpha |A|$ and $|B'| \geq \alpha |B|$. 
Then $(A', B')$ is an $\varepsilon'$-regular pair with density $d'$ where $|d'-d| < \varepsilon$.
\end{fact}

\begin{lem}[Key lemma \cite{KomlosSimonovits}]\label{keylem}
Suppose that $0 < \varepsilon < d$, that $q, t \in \mathbb{N}$ and that $R$ is a graph with $V(R) = \{v_1, \ldots, v_k\}$. 
We construct a graph $G$ as follows: Replace every vertex $v_i \in V(R)$ with a set $V_i$ of $q$ vertices and replace each edge of $R$ with an $\varepsilon$-regular pair of density at least $d$. For each $v_i \in V(R)$, let $U_i$ denote the set of $t$ vertices in $R(t)$ corresponding to $v_i$.
Let $H$ be a subgraph of $R(t)$ with maximum degree $\Delta$ and set $h := |H|$. Set $\delta := d - \varepsilon$ and $\varepsilon_0 := \delta^{\Delta}/(2 + \Delta)$.
If $\varepsilon \leq \varepsilon_0$ and $t-1 \leq \varepsilon_0q$ then there are at least $$(\varepsilon_0 q)^h \ \mbox{labelled copies of $H$ in $G$} $$ so that if $x \in V(H)$ lies in $U_i$ in R(t), then $x$ is embedded into $V_i$ in $G$.
\end{lem}

Let $G$ be a graph as in Theorem \ref{mainthmsigma} and $R$ a reduced graph of $G$. The next well known lemma essentially says that $R$ `inherits' the degree sequence of $G$.

\begin{lem}[See e.g. \cite{hlt}]\label{inheriteddegseqreduced}
Set $M', n_0 \in \mathbb{N}$ and $\varepsilon, d, \eta, b, \omega, \sigma$ to be positive constants such that $1/n_0 \ll 1/M' \ll \varepsilon \ll d \ll \eta, 1/b$ and where $\omega + \sigma \leq b$. 
Suppose $G$ is a graph on $n \geq n_0$ vertices with degree sequence $d_1 \leq \ldots \leq d_n$ such that
\begin{equation*} 
d_i \geq \frac{b-\omega-\sigma}{b}n + \frac{\sigma}{\omega}i + \eta n \ \ \mbox{for all $1 \leq i \leq \frac{\omega n}{b}$.}
\end{equation*}
Let $R$ be the reduced graph of $G$ with parameters $\varepsilon$, $M'$ and $d$ and set $k:= |R|$. Then $R$ has degree sequence $d_{R,1} \leq \ldots \leq d_{R,k}$ such that
\begin{equation*} 
d_{R,i} \geq \frac{b-\omega-\sigma}{b}k + \frac{\sigma}{\omega}i + \frac{\eta k}{2} \ \ \mbox{for all $1 \leq i \leq \frac{\omega k}{b}$.}
\end{equation*}
\end{lem}
Let $G$ and $H$ be graphs and $R$ be a reduced graph of $G$. Let $\mathcal{H}$ be a perfect $H$-tiling in $R$. The following result ensures that after removing only a few vertices from each cluster in $R$ every edge in each copy of $H \in \mathcal{H}$ corresponds to a superregular pair. This will be essential to apply Lemma \ref{blowup} in Section~\ref{sectblowup}.  

% \begin{fact} \label{superfact}
% Let $0 < \varepsilon < 1/3$ and $d \in [0,1]$. If $(A,B)$ is an $\varepsilon$-regular pair with density $d$, then there exists $A' \subseteq A$ and $B' \subseteq B$ such that $|A'| \geq (1 - \varepsilon)|A|$ and 
% $|B'| \geq (1 - \varepsilon)|B|$, and such that $(A',B')$ is a $(2\varepsilon,d - 3\varepsilon)$-superregular pair.
% \end{fact}

\begin{prop}[See e.g. \cite{kotaraz}]\label{superprop} Let $G$ be a graph, $\varepsilon, d \in (0,1)$ and $M',\Delta \in \mathbb{N}$. Apply Lemma~\ref{degformreglemma} to $G$ with parameters $\varepsilon, M'$ and $d$ to obtain a reduced graph $R$ with clusters of size $q$. Let $H$ be a subgraph of the reduced graph $R$ with $\Delta(H) \leq \Delta$ and label the vertices of $H$ as $V_1, \ldots, V_{|H|}$. Then each vertex $V_i$ of $H$ contains a subset $V_i'$ of size $(1 - \varepsilon \Delta)q$ such that for every edge $V_iV_j$ of $H$ the graph $(V_i',V_j')_{G'}$ is $(\varepsilon/(1 - \varepsilon\Delta),d - (1+\Delta)\varepsilon)$-superregular.
\end{prop}

The following fundamental result of Koml\'{o}s, S\'{a}rk\"{o}zy and Szemer\'{e}di~\cite{kssblowup}, known as the \emph{Blow-up lemma}, essentially says that $(\varepsilon, d)$-superregular pairs behave like complete bipartite graphs with respect to containing bounded degree subgraphs.

\begin{lem}[Blow-up lemma \cite{kssblowup})] \label{blowup} Given a graph $F$ on vertices $\{1,\ldots, f\}$ and $d, \Delta > 0$, there exists an $\varepsilon_0 = \varepsilon_0(d,\Delta,f) > 0$ such that the following holds.
Given $L_1, \ldots , L_f \in \mathbb{N}$ and $\varepsilon \leq \varepsilon_0$, let $F^{*}$ be the graph obtained from $F$ by replacing each vertex $i \in F$ with a set $V_i$ of $L_i$ new vertices and joining all vertices in $V_i$ to all vertices in $V_j$ whenever $ij$ is an edge of $F$. 
Let $G$ be a spanning subgraph of $F^{*}$ such that for every edge $ij \in F$ the pair $(V_i,V_j )_G$ is $(\varepsilon, d)$-superregular. Then $G$ contains a copy of every subgraph $H$ of $F^{*}$ with $\Delta(H) \leq \Delta$.
\end{lem}

\subsection{Tilings in complete graphs}
In \cite{kuhn}, the following result of K\"{u}hn and Osthus is essential to their proof of Theorem \ref{kothm}.

\begin{lem}\cite[Lemma 12]{kuhn2} \label{kolem}
Let $H$ be a graph with $\chi(H) =: r \geq 2$ and $\textnormal{hcf}(H) = 1$. Let $h := |H|$ and $\omega(H) := \left(h - \sigma(H)\right)/(r-1)$. Let $0 < \beta_1 \ll \lambda_1 \ll \sigma(H)/\omega(H)$, $1 - \sigma(H)/\omega(H)$, $1/h$ be positive constants. 
Suppose that $F$ is a complete $r$-partite graph with vertex classes $U_1, \ldots , U_{r}$ such that: $0 < 1/|F| \ll \beta_1$;
\COMMENT{AT: deleted some terms that were not needed}
 $|F|$ is divisible by $h$; $(1~-~\lambda_1^{1/10})|U_{r}|~\leq \sigma(H)|U_i|/\omega(H)~\leq~(1~-~\lambda_1)|U_{r}|$ for all $i < r$; $| |U_i | - |U_j | | \leq \beta_1|F|$ whenever $1 \leq i < j < r$. Then $F$ contains a perfect $H$-tiling.
\end{lem}

We will use the Blow-up lemma in tandem with the following generalisation of Lemma \ref{kolem} to conclude that a particular tiling that we construct in a reduced graph $R$ guarantees a perfect $H$-tiling in our original graph $G$.

\begin{lem} \label{generalkolem}
Let $H$ be a graph with $\chi(H) =: r \geq 2$ and $\textnormal{hcf}(H) = 1$.  Let $h := |H|$. Set $\sigma \in \mathbb{R}$ such that $\sigma(H) \leq \sigma < h/r$ and $\omega := \left(h - \sigma\right)/(r-1)$. Let $0 < \beta_2 \ll \lambda_2 \ll \sigma/\omega$, $1 - \sigma/\omega$, $1/h$ be positive constants. 
Suppose that $F$ is a complete $r$-partite graph with vertex classes $U_1, \ldots , U_{r}$ such that: $0 < 1/|F| \ll \beta_2$; 
$|F|$ is divisible by $h$; $(1 - \lambda_2^{1/10})|U_{r}| \leq \sigma|U_i|/\omega \leq (1 - \lambda_2)|U_{r}|$ for all $i < r$; $| |U_i | - |U_j | | \leq \beta_2|F|$ whenever $1 \leq i < j < r$. Then $F$ contains a perfect $H$-tiling.
\end{lem}
\proof
Note we may assume that $\sigma >\sigma (H)$ as otherwise the result follows immediately from Lemma~\ref{kolem}.
We choose $\beta _2 \ll \beta _1 \ll \lambda _2 \ll \lambda _1$ where $\beta _1$ and $\lambda _1$ are as in Lemma~\ref{kolem}.
Additionally we may assume $\beta _2, \lambda _2 \ll (\sigma/\omega-  \sigma(H)/\omega(H))$.

Let $F$ be as in the statement of the lemma.
Set $H^*$ to be the complete balanced $r$-partite graph on $rh$ vertices (that is, each vertex class of $H^*$ has size $h$). Observe that $H^*$ has a perfect $H$-tiling using precisely $r$ copies of $H$.

Repeatedly delete disjoint copies of $H^*$ from $F$ (and therefore update the classes $U_1,\dots, U_r$) until the first point for which there is some $i<r$ such that 
$(1-\lambda_1^{1/10}/2)|U_{r}|\leq \sigma(H)|U_i|/\omega(H)\leq(1-2\lambda_1)|U_{r}|$. 
Call the resulting graph $F'$. Note that $\sigma /\omega > \sigma (H)/\omega (H)$, so we can indeed obtain $F'$. Further note that the choice of $\beta _2$ ensures each
class $U_j$ still contains at least a $\beta _2 ^{1/2}$-proportion of the vertices it started with.
So now  $| |U_i | - |U_j | | \leq \beta_2|F| \leq \beta ^{1/2}_2|F'|  \leq \beta _1 |F'|$ whenever $1 \leq i < j < r$.
Moreover, this implies 
$(1-\lambda_1^{1/10})|U_{r}|\leq \sigma(H)|U_j|/\omega(H)\leq(1-\lambda_1)|U_{r}|$ for all $j <r$.
Thus, by Lemma~\ref{kolem}, $F'$ contains a perfect $H$-tiling and therefore, so too does $F$, as desired.
\endproof
\COMMENT{AT: new short proof}

\subsection{A degree sequence Koml\'{o}s theorem}
 In \cite{kuhn}, K\"{u}hn and Osthus begin their proof of Theorem~\ref{kothm} by applying Koml\'{o}s' theorem (Theorem~\ref{komcor}). 
% \COMMENT{JH: I've taken out talking about the almost perfect $H$-tiling here because although we do get an almost perfect $H$-tiling that's only as a consequence of getting the almost perfect $\hat{B}$-tiling. I believe talking about `finding a $H$-tiling' here could be misleading and confusing.} 
In our proof of Theorem~\ref{mainthmsigma} we will use the following degree sequence version of Koml\'{o}s' theorem that the authors and Liu proved in \cite{hlt}. % our degree sequence version of Koml\'{o}s' theorem. We state it again here for reference.

\begin{thm} \cite[Theorem 8.1]{hlt} \label{generalalmostmain}
Let $\eta > 0$ and $H$ be a graph with $\chi(H) = r$ and $h := |H|$. Set $\sigma \in \mathbb{R}$ such that $\sigma(H) \leq \sigma \leq h/r$ and $\omega := \left(h - \sigma\right)/(r-1)$. Then there exists an integer $n_0 = n_0(\eta, \sigma, H) \in \mathbb{N}$ such that the following holds. Suppose $G$ is a graph on $n\geq n_0$ vertices with degree sequence $d_1 \leq \ldots \leq d_n$ such that

\[d_i \geq \left(1 - \frac{\omega+\sigma}{h}\right)n + \frac{\sigma}{\omega}i \ \ \mbox{for all \ $1 \leq i \leq \frac{\omega n}{h}$.}\]
Then $G$ contains an $H$-tiling covering all but at most $\eta n$ vertices.
\end{thm}

\COMMENT{AT: deleted Chernoff here as we no longer use it explicitly in paper}
\subsection{B\'ezout's Lemma}
To prove Theorem~\ref{partitionthm} we will need the following elementary arithmetic result.

\begin{lem}[B\'ezout's Lemma]\label{bezoutlemma}
Let $a_1,a_2, \ldots, a_t \in \mathbb{Z}$. Then there exist integers $y_1, y_2, \ldots, y_t \in \mathbb{Z}$ such that 
$$\sum_{i=1}^{t} y_ia_i = \textnormal{hcf}(a_1, a_2, \ldots, a_t)$$
where \textnormal{hcf}$(a_1, a_2, \ldots, a_t)$ is the highest common factor of $a_1, a_2, \ldots, a_t$.
\end{lem}

%\section{Proof sketch of Lemma \ref{generalkolem}}\label{sectkolem}

%Let $\beta_2$ be a constant such that $\beta_1 \ll \beta_2 \ll \lambda_1$. We can remove a tiling $\mathcal{H}$ of copies of $H$ from $F$ in order to produce a complete $r$-partite graph $F'$ with vertex classes $U_1', \ldots, U_r'$ such that: $0 < 1/|F'| \ll \beta_2, \lambda_1, \sigma(H)/\omega(H), 1 - \sigma(H)/\omega(H), 1/h$; $|F'|$ is divisible by $h$; $(1 - \lambda_1^{1/10})|U_{r}'| \leq \sigma(H)|U_i'|/\omega(H) \leq (1 - \lambda_1)|U_{r}'|$ for all $i < r$; $| |U_i' | - |U_j' | | \leq \beta_2|F'|$ whenever $1 \leq i < j < r$. 
%We then apply Lemma \ref{kolem} to deduce that $F'$ has a perfect $H$-tiling $\mathcal{H}'$. Then $\mathcal{H} \cup \mathcal{H}'$ is a perfect $H$-tiling in $F$. See `Appendix: Proof of Lemma \ref{generalkolem}' for a complete proof.
% \COMMENT{One could truncate the working here by a considerable margin, but I think it would be useful to keep for the reader. There aren't any clever things in the proof, its just fiddling about for a couple of pages. So for the sake of the reader I think one could keep the first paragraph, which describes the proof idea, and put the rest in an appendix?}

\section{A tool for the proof of Theorem~\ref{mainthmsigma}}\label{sectstandalone}

In this section, we prove a theorem (Theorem \ref{partitionthm}) that will be used in Sections~\ref{sectcase1} and~\ref{sectcase2} of the proof of~Theorem~\ref{mainthmsigma}. At the beginning of Section~\ref{sectdivisible}, we will have a certain $\hat{B}$-tiling $\hat{\mathcal{B}}$ of a reduced graph $R$ (the graph $\hat{B}$ will be defined later). Denote the copies of $\hat{B}$ in $\hat{\mathcal{B}}$ by $\hat{B}_1, \hat{B}_2, \ldots, \hat{B}_{\hat{k}}$. For applications of Lemma~\ref{generalkolem} required at the end of our proof of Theorem~\ref{mainthmsigma}, we will need $|V_G(\hat{B}_i)|$ to be divisible by $h$ for each $1 \leq i \leq \hat{k}$. The following theorem is the crucial tool for ensuring 
we can remove copies of $H$ from $G$ to achieve this.\COMMENT{AT: edited sentence}

For a graph $H$ with $\chi(H) = r$, recall that $C_H$ is the set of all optimal colourings of $H$ and that given an optimal colouring $c \in C_H$ we let $x_{c,1} \leq x_{c,2} \leq \ldots \leq x_{c,r}$ denote the sizes of the colour classes of $c$. We require the following definitions.

\begin{define}\label{defy}
Let $H$ be a graph with $\chi(H) =: r$. Fix $1 \leq p \leq r-1$. For each $c \in C_H$, define $D_c$ to be the multiset $[x_{c, 1} ,x_{c, 2}, \ldots, x_{c, r}]$. We say that $A$ is a \emph{$p$-subset contained in $D_c$} if $A$ is a multiset, $|A| = p$ and $A = [x_{c, j_1}, x_{c, j_2}, \ldots, x_{c, j_p}]$ where $j_1,j_2, \ldots, j_p \in \{1, \ldots, r\}$ are distinct. Let $z_p := \rp$ be the number of $p$-subsets contained in $D_c$. For each colouring $c \in C_H$, 
label the $p$-subsets contained in $D_c$ by $A_{p,c,1}, A_{p,c,2}, \ldots, A_{p,c,z_p}$. Let $S_{p,c,J} := \sum_{x \in A_{p,c,J}}x$ for each $c \in C_H$, $1 \leq J \leq z_p$.
\end{define}

\begin{thm}\label{partitionthm}
Let $H$ be an $r$-partite graph  and let $h := |H|$. Fix $1 \leq p \leq r-1$. Let $b$ be the number of components of $H$ and $t_1, \ldots, t_b$ be the sizes of the components of $H$. Then 
\begin{itemize} 
\item if $r = 2$ and $\textnormal{hcf}_c(H) = 1$,\COMMENT{AT: made it explicit in the statement of the theorem that we only require $\textnormal{hcf}_c(H) = 1$ as that helps readers understanding, and also we mention this later}
 there exists a collection of non-negative integers $\{a_{i}:  1 \leq i \leq b\}$ and $\bar{a} \in \mathbb{N}$ such that $$a_i \leq \bar{a} \ \ \mbox{for all} \ \ 1 \leq i \leq b,$$  and $$\sum_{i = 1}^{b}a_it_i \equiv 1\mod h.$$ 

\item if $r \geq 3$ and $\textnormal{hcf}_{\chi}(H) = 1$, there exists a collection of non-negative integers $\{a_{p,c,i}: c \in C_H, 1 \leq i \leq z_p\}$ and $\bar{a} \in \mathbb{N}$ such that 
    $$a_{p,c,i} \leq \bar{a} \ \ \mbox{for all} \ c \in C_H \ \mbox{and} \ 1 \leq i \leq z_p,$$ and
    $$\sum_{c \in C_H}\sum_{i = 1}^{z_p}a_{p,c,i}S_{p,c,i} \equiv 1\mod h.$$
\end{itemize}
\end{thm}

%It should be noted that Theorem~\ref{partitionthm} is a standalone elementary number theoretical result and only requires that $\textnormal{hcf}(H) = 1$.
\COMMENT{AT:deleted sentence as I'm not sure it adds anything.}

For each $1 \leq p \leq r - 1$, $c \in C_H$ and $j \in \{1, \ldots, r\}$, let $Z_{p,c,j}$ be the multiset defined by the following table:

\begin{center}
\begin{tabular}{ c|c|c|c|c|c|c|c } 
 
 Colour class size          & $x_{c, 1}$    &  $\cdots$       &  $x_{c, j-1}$     &  $x_{c, j}$       & $x_{c, j+1}$  &  $\cdots$             & $x_{c, r}$    \\ 
 \hline
 Multiplicity in $Z_{p,c,j}$  & $p$             &  $\cdots$       &  $p$                &  $p + 1$            & $p$             &  $\cdots$             & $p$             \\ 

\end{tabular}
\end{center}

The following fact will be useful in our proof of Theorem~\ref{partitionthm}.

\begin{fact}\label{zfact}
For any $1 \leq J,L \leq r$, we can partition $Z_{p,c,J}$ into $\{x_{c, L}\}$ and $r$ $p$-subsets contained in $D_c$.
\end{fact}

\smallskip
\begin{proofofpartitionthm}
Firstly, we will consider the case when $r = 2$ and $\textnormal{hcf}_c(H) = 1$. So $H$ must have multiple components. The sizes of these components of $H$ are $t_1, t_2, \ldots, t_b$.  
Since $\textnormal{hcf}_{c}(H) = 1$, by Bezout's Lemma (Lemma~\ref{bezoutlemma}) there exist integers $a'_1, \ldots, a'_b$ such that $$\sum_{i = 1}^{b}a'_it_i = \textnormal{hcf}(t_1, \ldots, t_b) = 1.$$ Since $\sum_{i = 1}^{b}t_i = h$, there exists $\hat{a} \in \mathbb{N} \cup \{0\}$ such that
$$\sum_{i = 1}^{b}(a'_i + \hat{a})t_i \equiv 1\mod h$$ and $$a'_i + \hat{a} \geq 0 \ \ \mbox{for all} \ \ 1 \leq i \leq b.$$ For each $1 \leq i \leq b$, take $a_i := a'_i + \hat{a}$ and $\bar{a} := \max\limits_{i=1, \ldots,b} \ a_i$. 

\smallskip
Next consider when $r \geq 3$. Instead of explicitly calculating $a_{p,c,i}$ for each $c \in C_H$, $1 \leq i \leq z_p$, we will construct for each $c \in C_H$ a multiset $X_c$ of bounded size which can be partitioned into $p$-subsets contained in $D_c$. Further, we will construct our multisets $X_c$ such that
$$\sum_{c \in C_H}\sum_{x \in X_c} x \equiv 1\mod h.$$ 
Observe that constructing such multisets $X_c$ immediately yields a collection of non-negative integers $\{a_{p,c,i}: c \in C_H, 1 \leq i \leq z_p\}$ that satisfy the conditions in Theorem~\ref{partitionthm}. Indeed, 
for each $c \in C_H$ and $1\leq i \leq z_p$, we take $a_{p,c,i}$ to be precisely the number of times $A_{p,c,i}$ occurs in the partition of $X_c$ into $p$-subsets.
\COMMENT{AT: perhaps what was written before `over explains' things. I think the reader should be see what we mean with this new paragraph}

%for each $c \in C_H$, label the $p$-subsets contained in $D_c$ in the partition of $X_c$ as $B_{c, 1}, \ldots, B_{c, |X_c|/p}$. Then for each $c \in C_H$, $i \in \{1, \ldots, |X_c|/p\}$ there exists $j \in \{1, \ldots, z_p\}$ such that $B_{c, i} = A_{p,c,j}$. Thus there exists a collection of non-negative integers $\{a_{p,c,i}: c \in C_H, 1 \leq i \leq z_p\}$ such that $$\sum_{c \in C_H}\sum_{j = 1}^{z_p}a_{p,c,j}S_{p,c,j} = \sum_{c \in C_H}\left(\sum_{i=1}^{|X_c|/p}\sum_{x \in B_{c,i}} x\right) = \sum_{c \in C_H}\sum_{x \in X_c} x.$$  Moreover, since the multisets $X_c$ we will construct will be of bounded size, we will have that there exists $\bar{a} \in \mathbb{N}$ such that $a_{p,c,i} \leq \bar{a}$ for all $c \in C_H, 1 \leq i \leq z_p$.

In order to start constructing our multisets $X_c$, we define the following multiset: $$\mathcal{D}^{*}(H) := \underset{c \in C_H}{\bigcup} [x_{c, j+1} - x_{c, j} \  | \ j = 1, \ldots, r-1].$$

As $\textnormal{hcf}_{\chi}(H) = 1$ we can apply Lemma~\ref{bezoutlemma} to the multiset $\mathcal{D}^{*}(H)$ to get for each $c \in C_H$, $1 \leq j \leq r-1$ integers $b_{c, j}$ such that the following holds:
\begin{equation} \label{bezouts}
\sum_{c \in C_H}\sum_{j=1}^{r-1} b_{c, j}( x_{c, j+1} - x_{c, j} ) \equiv 1\mod h.
\end{equation}
% We have used the colour class differences ($x_{c, j+1} - x_{c, j}$) for each $c \in C_H$, $1 \leq j \leq r-1$ in the writing of (\ref{bezouts}) to illustrate that it is the number of each colour class present in these differences that we are interested in, rather than the values of each difference.
% \footnote{I currently feel that trying to explain why we are explicitly writing out the differences rather than assigning values to the differences (like $x_{c, j+1} - x_{c, j} = z_{c,j}$) is more confusing than leaving it without explanation.}

We now construct our multisets $X_c$. Fix $c \in C_H$. Let $t_c \in \mathbb{N}$ be the smallest $t' \in \mathbb{N}$ such that $$pt' \geq \textnormal{max}\{ | b_{c, 1} |, | b_{c, 1} - b_{c, 2} |, | b_{c, 2} - b_{c, 3} |, \ldots, |b_{c, r-2} - b_{c, r-1}|, |b_{c, r-1}|\}.$$ 
Then $pt_c - b_{c, 1} , pt_c + b_{c, 1} - b_{c, 2} ,  pt_c + b_{c, 2} - b_{c, 3} ,\ldots,  pt_c + b_{c, r-2} - b_{c, r-1} , pt_c + b_{c, r-1}$ are non-negative integers. 
Let $Y_c$ be the multiset defined by the following table:

\begin{center}
\begin{tabular}{ m{3cm} | m{5em} | m{5em} | m{5em} | m{1em} | m{6em} | m{5em} }
 
 Colour class size         & $x_{c, 1}$            & $x_{c, 2}$                        & $x_{c, 3}$                            & $\cdots$    & $x_{c, r-1}$                          & $x_{c, r}$            \\ 
 \hline
 Multiplicity in $Y_c$  & $pt_c - b_{c, 1}$      & $pt_c + b_{c, 1} - b_{c, 2}$     & $pt_c + b_{c, 2} - b_{c, 3}$         & $\cdots$    & $pt_c + b_{c, r-2} - b_{c, r-1}$     & $pt_c + b_{c, r-1}$    \\ 

\end{tabular}
\end{center}
Then $|Y_c| = rpt_c$. If we can partition $Y_c$ into $p$-subsets contained in $D_c$ then we take $X_c := Y_c$.
Assume we cannot. Then the multiplicities of $x_{c,1}, \ldots, x_{c,r}$ in $Y_c$ must be sufficiently different from one another. We employ the following algorithm which transforms $Y_c$ into a multiset which can be partitioned into $p$-subsets contained in $D_c$ using Fact \ref{zfact}. To state the algorithm we require the following definition.

\begin{define} \textnormal{For each $c \in C_H$, $1 \leq i \leq r$, let $m_{c,i}$ be the multiplicity of $x_{c,i}$ in $Y_c$. 
Let $J,L \in \{1, \ldots, r\}$ such that $J \neq L$; $m_{c,J} \geq \frac{\sum_{i = 1}^{r}m_{c,i}}{r}$; $m_{c,L} \leq \frac{\sum_{i = 1}^{r}m_{c,i}}{r}$; $m_{c,L} + 1 \neq m_{c,J}$; $m_{c,L} \neq m_{c,J}$. Let $Y_c' := Y_c - \{x_{c, J}\} + \{x_{c, L}\}$.\footnote{That is, $Y_c'$ is the multiset $Y_c$ except with $x_{c,J}$ having multiplicity $m_{c,J} - 1$ and $x_{c,L}$ having multiplicity $m_{c,L} + 1$.} 
Then we say that $Y_c'$ is \emph{more balanced than $Y_c$}.}
\end{define}

\begin{alg} \ \begin{itemize}
    \item[1)] Let $Q := \emptyset$.
    \item[2)] If $|m_{c,i} - m_{c,j}| = 0$ for all $1 \leq i,j \leq r$, output $Y_c$ and $Q$. Otherwise, choose $J, L \in \{1, \ldots, r\}$ such that $Y_c' := Y_c - \{x_{c, J}\} + \{x_{c, L}\}$ is more balanced than $Y_c$.
    \item[3)] Add $p$ copies of $H$ with colouring $c$ to $Y_c$. That is, $x_{c,i}$ now has multiplicity $m_{c,i} + p$ in $Y_c$ for each $1 \leq i \leq r$.
    \item[4)] Take $Z_{p,c,J}$ to be the union of $\{x_{c,J}\}$ and these $p$ copies of $H$. By Fact~\ref{zfact} there exists a partition of $Z_{p,c,J}$ into $\{x_{c,L}\}$ and $r$ $p$-subsets contained in $D_c$.
    \item[5)] Place into $Q$ these $r$ $p$-subsets contained in $D_c$.
    \item[6)] Take $Y_c := Y_c'$ and update the value of each $m_{c,i}$ (that is, $m_{c,J}$ has decreased by 1 and $m_{c,L}$ has increased by 1). Go to Step 2.
\end{itemize} 
\end{alg}

Therefore, at the end of the algorithm $|Y_c| = rpt_c$ and $|m_{c,i} - m_{c,j}| = 0$ for all $1 \leq i,j \leq r$.
In particular, it is easy to see
  that $Y_c$ now has a partition $Q_{Y_c}$ into $p$-subsets contained in $D_c$. % Then we take the partition \{x_{c,1}, x_{c,2}, \ldots, x_{c,p}\}, \{x_{c,p+1}, \ldots, x_{c,2p}\}, \ldots, \{x_{r-i}, \ldots, x_r, \ldots, x_{p-(i+1)} }\} 
Let $\hat{t}_c$ be the number of collections of $p$ copies of $H$ added during the algorithm and $\hat{t}_c := t_c + \hat{t}_c$. Then the multiset $\hat{Y}_c$, defined by the table below, can be partitioned into $p$-subsets contained in $D_c$ using the partition $Q \cup Q_{Y_c}$:

\begin{center}
\begin{tabular}{ m{3cm} | m{5em} | m{5em} | m{5em} | m{1em} | m{6.2em} | m{5em} } 
 
 Colour class size          & $x_{c, 1}$                    & $x_{c, 2}$                                & $x_{c, 3}$                                    & $\cdots$    & $x_{c, r-1}$                                  & $x_{c, r}$                    \\ 
 \hline
 Multiplicity in $\hat{Y}_c$  & $p\hat{t}_c - b_{c, 1}$      & $p\hat{t}_c + b_{c, 1} - b_{c, 2}$     & $p\hat{t}_c + b_{c, 2} - b_{c, 3}$         & $\cdots$    & $p\hat{t}_c + b_{c, r-2} - b_{c, r-1}$     & $p\hat{t}_c + b_{c, r-1}$    \\ 

\end{tabular}
\end{center}

Take $X_c := \hat{Y}_c$. We now confirm that our multisets $X_c$ satisfy  $$\sum_{c \in C_H}\sum_{x \in X_c} x \equiv 1\mod h.$$
By (\ref{bezouts}) and the definition of $X_c$ for each $c \in C_H$ we have
\begin{align*}
& \sum_{c \in C_H}\sum_{x \in X_c} x 
\\ = & \sum_{c \in C_H}\left (\sum_{j=1}^{r-1} b_{c, j}( x_{c, j+1} - x_{c, j} ) + p\hat{t}_c\left ( \sum_{j = 1}^{r} x_{c, j} \right ) \right ) 
& \\ = & \left ( \sum_{c \in C_H}\sum_{j=1}^{r-1} b_{c, j}( x_{c, j+1} - x_{c, j} )\right ) + \left (p\sum_{c \in C_H}\hat{t}_c\right )h 
& \\ \overset{(\ref{bezouts})}{\equiv} & \ 1\mod h.
\end{align*}
Therefore, recalling the discussion earlier in this proof, there must exist the desired collection of non-negative integers $\{a_{p,c,i}: c \in C_H, 1 \leq i \leq z_p\}$. 
%and $\bar{a} \in \mathbb{N}$ that satisfy the conditions of Theorem~\ref{partitionthm}.
\COMMENT{AT: don't think we have to explicitly mention $\bar{a}$ here as its existence is just because every finite set of numbers has a maximum}
\end{proofofpartitionthm}\qed

\section{Proof Overview}\label{sectoverview}

The rest of this paper  % chapter
will be devoted to proving Theorem~\ref{mainthmsigma} and here we outline the proof. As noted in Section~\ref{sectdegseqcond}, our proof follows closely K\"{u}hn and Osthus' proof of Theorem~\ref{kothm} in \cite{kuhn}. 

Let $H$, $G$, $\eta$ and $\sigma$ be as in the statement of the theorem. 
In particular, $h:=|H|$ and $\omega :=(h-\sigma)/(r-1)$.
Note that it suffices to prove the result in the case when $\sigma \in \mathbb{Q}$.
First we define a bottle graph $B$ that contains a perfect $H$-tiling.
 \begin{define} Let  $a,b \in \mathbb{N}$ such that $\sigma = a/b$. Let $\omega(H) := (h - \sigma(H))/(r-1)$ and $\hat{c} := b(r-1)(\omega(H) - \sigma(H))$.
Define $B$ to be the $r$-partite bottle graph with neck $\sigma \hat{c}$ and width $\omega \hat{c}$ (note that $\sigma \hat{c}, \omega \hat{c} \in \mathbb N$). Observe that $|B| = h\hat{c}$; $\sigma(B) = \sigma \hat{c}$; $\omega(B) = \omega \hat{c}$.
 Further,
\begin{equation*}
\chi_{cr}(B) = r - 1 + \sigma/\omega=h/\omega.
\end{equation*}
\end{define}
Since $|B| = h\hat{c}$; $\sigma(B) = \sigma \hat{c}$; $\omega(B) = \omega \hat{c}$, we have that $G$ satisfies the hypothesis of the degree sequence Koml\'{o}s theorem (Theorem~\ref{generalalmostmain})
with $B$, $\sigma (B)$ and $\omega (B)$ playing the roles of $H$, $\sigma$ and $\omega$ respectively. That is, $G$ contains an almost perfect $B$-tiling.
In fact, as the reduced graph $R$ of $G$ almost inherits the degree sequence of $G$, Theorem~\ref{generalalmostmain} ensures that $R$ contains an almost perfect $B$-tiling $\mathcal B$.
Further note that the choice of $\hat{c}$ implies that $B$ has a perfect $H$-tiling consisting of $\hat{c}$ copies of $H$. (A simple proof of this can be found inside the proof of Theorem~8.1 in~\cite{hlt}.)

Ideally one would like to use $\mathcal B$ as a framework to build  the perfect $H$-tiling in $G$. However, as explained shortly, we need more flexibility in our tiling in $R$.
Therefore, we introduce the following `modified' version of $B$.

 \begin{define}
Let $s \in \mathbb{N}$ be sufficiently large and $\lambda \in \mathbb{R}^{+}$ be sufficiently small.
Recall that $\sigma < \omega$. Let $\hat{B}$ be the $r$-partite bottle graph with neck $\sigma(1 + \lambda)s/\omega$ and width $s$.\footnote{We have that $\sigma (1 + \lambda)/\omega < 1$ by our choice of $\lambda$ and that $\sigma < \omega$.}
Moreover, we choose $\lambda$ and $s$ such that $\hat{B}$ contains a perfect $B$-tiling. Hence $\hat{B}$ contains a perfect $H$-tiling.
Note that \begin{equation*}
\chi_{cr}(\hat{B}) = r - 1 + \sigma(1 + \lambda)/\omega. 
\end{equation*} \end{define}

Since  $\lambda$ is chosen to be small (and so $\chi_{cr}(\hat{B})$ is very close to $\chi_{cr}(B)$), one can still  apply Theorem~\ref{generalalmostmain} on inputs $\hat{B}$ and $R$. That is, 
$R$ contains  an almost perfect $\hat{B}$-tiling $\mathcal{\hat{B}}$.  Denote the copies of $\hat{B}$ in $\hat{\mathcal{B}}$ by $\hat{B}_1, \hat{B}_2, \ldots, \hat{B}_{\hat{k}}$. By removing a small number of vertices from each cluster in $R$, we can ensure the edges of each $\hat{B}_i$ correspond to superregular pairs. Let $V_0$ denote the set of all vertices in $G$ not contained in the clusters lying in the tiling $\mathcal{\hat{B}}$.

For each $1 \leq i \leq \hat{k}$, let $\hat{G}_i$ be the $r$-partite subgraph of $G$ whose $j$th vertex class is the union of all those clusters contained in the $j$th vertex class of $\hat{B}_i$, for each $1 \leq j \leq r$. Let $G^{*}_i$ be the complete $r$-partite graph on the same vertex set as $\hat{G}_i$. 
%For a subgraph $S \subseteq R$, let $V_G(S)$ denote the union of the clusters in $S$. 
We introduce the graph $\hat{B}$ (rather than just working with $B$) since $\hat{B}$ has the following crucial property: For each $1 \leq i \leq \hat{k}$ we can arbitrarily delete a small number of vertices from $G^{*}_i$
(and correspondingly $\hat{G}_i$)
 and, provided $|V(G^{*}_i)|$ is now divisible by $h$, the resulting graph satisfies the hypothesis of Lemma~\ref{generalkolem}. That is, this graph contains a perfect $H$-tiling. Then the Blow-up lemma (Lemma~\ref{blowup}) implies that each $\hat{G}_i$ contains a perfect $H$-tiling.

We make use of this property of $\hat{B}$ as follows: In Section~\ref{sectabsorb} we remove copies of $H$ from $G$ that cover all vertices in $V_0$, as well as a small (possibly zero) number of vertices from each $\hat{G}_i$;
call this $H$-tiling (formed from these copies of $H$) $\mathcal{H}_1$. Deleting these covered vertices from each $\hat{G}_i$, if $|V(\hat{G}_i)| \ (= |V(G^*_i)|)$ is still divisible by $h$ for each $1 \leq i \leq \hat{k}$ then each $\hat{G}_i$ now contains a perfect $H$-tiling (by our argument above). However, for some $i$, we may have that $|V(\hat{G}_i)|$ is not divisible by $h$.
So in Section~\ref{sectdivisible} we remove a further bounded number of copies of $H$, forming an $H$-tiling $\mathcal{H}_2$, to ensure $|V(\hat{G}_i)| \ (= |V(G^*_i)|)$ is divisible by $h$ for each $1 \leq i \leq \hat{k}$.
Thus, we now have that each $\hat{G}_i$ contains a perfect $H$-tiling $\hat{\mathcal{H}}_i$. Combining the tilings $\mathcal{H}_1, \mathcal{H}_2, \hat{\mathcal{H}}_1, \ldots, \hat{\mathcal{H}}_{\hat{k}}$ yields a perfect $H$-tiling in $G$, as desired.

% In Section \ref{sectabsorb} we remove copies of $H$ from $G$ which contain all vertices from $V_0$ and a number of vertices from $\cup_{i = 1}^{\hat{k}}V_G(\hat{B}_i)$.
% Removing these copies of $H$ may make $|V_G(\hat{B}_i)| (= |V(G^{*}_i)|)$ not divisible by $h$ for some values of $i$. This prevents us from immediately applying Lemma~\ref{kolem} to each $G^{*}_i$ and using Lemma~\ref{blowup} to conclude that $G$ contains a perfect $H$-tiling.
%Hence, in Section \ref{sectdivisible} we remove a further bounded number of copies of $H$ to ensure that $|V_G(\hat{B})|$ is divisible by $h$ for each $\hat{B} \in \mathcal{\hat{B}}$. 
Our argument in Section~\ref{sectdivisible} will split into two cases, the first being when $\chi(H) \geq 3$ and the latter when $H$ is bipartite. 
This is where our proof differs the most from that in~\cite{kuhn} as we must make use of 
 Theorem~\ref{partitionthm} to find suitable copies of $H$ to ensure each $|V(\hat{G}_i)|$ is  divisible by $h$. 
\COMMENT{AT: The start of the proof overview was too dense. Tried to help the reader more.}
\COMMENT{AT: bit too strong to say our argument significantly differs from theirs here, as the Claims you cite, do look a lot like ours!}

% We conclude our proof in Section \ref{sectblowup}: Recall that $\hat{B}$ was chosen so that, for any $1 \leq i \leq \hat{k}$, after the removal of a small amount vertices from $V_G(\hat{B}_i)$ one can still apply Lemma~\ref{kolem} to $G^{*}_i$. Hence, for each $i$, we apply Lemma~\ref{kolem} to $G^{*}_i$ and then apply Lemma~\ref{blowup} to get a perfect $H$-tiling in $\hat{G}_i$.\footnote{We apply Lemma~\ref{blowup} when $F$ is $\hat{B}_i$; $F^{*}$ is $G^{*}_i$; $G$ is $\hat{G}_{i}$; and the subgraph of $F^{*}$ we want $G$ to contain is the perfect $H$-tiling of $G^{*}_i$ guaranteed by Lemma~\ref{kolem}.} Combining these $H$-tilings with the copies of $H$ we removed earlier, we have a perfect $H$-tiling in $G$, as desired.

\section{Proof of Theorem~\ref{mainthmsigma}} \label{sectsetup}

\subsection{Applying the regularity lemma}\label{sectreglem}
Note that it suffices to prove the theorem in the case when $\sigma \in \mathbb Q$.
Let $n$ be sufficiently large  and fix constants that satisfy the following hierarchy 
\begin{equation}\label{hier}
0 < 1/n \ll 1/M' \ll \varepsilon \ll d \ll \eta_1 \ll \beta \ll \alpha \ll \lambda \ll \eta, \sigma/\omega, 1 - \sigma/\omega, 1/h.\COMMENT{JH: Should we place $s$ or $\lambda_1$ (from Lemma \ref{generalkolem} in this hierarchy? Kuhn and Osthus don't put either in.}
\end{equation} 
As discussed in Section \ref{sectoverview}, we choose $s \in \mathbb{N}$ sufficiently large and define $\hat{B}$ to be the $r$-partite bottle graph with neck $\sigma(1 + \lambda)s/\omega$ and width $s$. As before, we choose $\lambda$ and $s$ such that $\hat{B}$ contains a perfect $B$-tiling, which implies that $\hat{B}$ contains a perfect $H$-tiling.
Note again that \begin{equation*}
\chi_{cr}(\hat{B}) = r - 1 + \sigma(1 + \lambda)/\omega. 
\end{equation*}
Moreover, choose % $\beta$, 
$\eta_1$ and $M'$ such that 
\begin{equation*}
\eta_1 \ll 1/|\hat{B}| \ \mbox{and} \ M' \geq n_0(\eta_1, \sigma (\hat{B}), \hat{B}),\COMMENT{AT: important that it is $\sigma (\hat{B})$ not just $\sigma$ here}
\end{equation*}
where $n_0$ is defined as in Theorem~\ref{generalalmostmain}. Let $G$ be an $n$-vertex graph as in the statement of Theorem~\ref{mainthmsigma}. Apply Lemma~\ref{degformreglemma} with parameters $\varepsilon$, $d$ and $M'$ to $G$ to obtain clusters $V_1, \ldots, V_k$ and an exceptional set $V_0$, where $q := |V_1| =  \ldots = |V_k|$ and $k \geq M'$.
Let $R$ be the corresponding reduced graph. Using (\ref{hier}), we may apply Lemma~\ref{inheriteddegseqreduced} with parameters $M', n, \varepsilon, d, \eta, h, \omega, \sigma$ to conclude that $R$ has degree sequence 
 $d_{R,1} \leq d_{R,2} \leq \ldots \leq d_{R,k}$ where
\begin{equation}\label{originalrdegseq}
d_{R,i} \geq \left(1 - \frac{\omega+\sigma}{h}\right)k + \frac{\sigma}{\omega}i + \frac{\eta k}{2} \ \ \mbox{for all $1 \leq i \leq \frac{\omega k}{h}$.}
\end{equation} 
For a graph $F$, recall that $\sigma(F)$ denotes the size of the smallest possible 
colour class in any $\chi(F)$-colouring of $F$ and $\omega(F) := (|F| - \sigma(F))/(\chi(F)-1)$.
Since $\lambda \ll \eta$, we have that  
\begin{equation} \label{rhatbdegseq}
d_{R,i} \geq \left(1 - \frac{\omega(\hat{B})+\sigma(\hat{B})}{|\hat{B}|}\right)k + \frac{\sigma(\hat{B})}{\omega(\hat{B})}i \ \ \mbox{for all $1 \leq i \leq \frac{\omega(\hat{B}) k}{|\hat{B}|}$.}
 \end{equation}

Since $|R| = k \geq M' \geq n_0(\eta_1, \sigma(\hat{B}), \hat{B})$ and (\ref{rhatbdegseq}) holds, we apply Theorem~\ref{generalalmostmain} to find a $\hat{B}$-tiling $\hat{\mathcal{B}}$ covering all but at most $\eta_1 k$ vertices in $R$. 
Denote the copies of $\hat{B}$ in $\hat{\mathcal{B}}$ by $\hat{B}_1, \hat{B}_2, \ldots, \hat{B}_{\hat{k}}$. 
Now delete all clusters not contained in some $\hat{B}_i$ from $R$ and add the vertices in these clusters to $V_0$. Therefore now $$|V_0| \leq \varepsilon n + \eta_1 n \leq 2\eta_1 n.$$
From now on, we denote by $R$ the subgraph of the reduced graph induced by all the remaining clusters and redefine $k := |R|$. Since $\eta_1 \ll \eta$, (\ref{originalrdegseq}) implies that $R$ has degree sequence $d_{R,1} \leq d_{R,2} \leq \ldots \leq d_{R,k}$ where
\begin{equation}\label{rdegseqnew}
 d_{R,i} \geq \left(1 - \frac{\omega+\sigma}{h}\right)k + \frac{\sigma}{\omega}i + \frac{\eta k}{4} \ \ \mbox{for all $1 \leq i \leq \frac{\omega k}{h}$.}
\end{equation}

For each $\hat{B}_i$ in $\hat{\mathcal{B}}$, let $\mathcal{B}_i$ be a perfect $B$-tiling in $\hat{B}_i$ (recall that earlier we chose $s$ and $\lambda$ such that $\hat{B}$ contains a perfect $B$-tiling). Let $\mathcal{B} := {\bigcup} \ \mathcal{B}_i$ and observe that $\mathcal{B}$ is a perfect $B$-tiling in $R$. To aid with calculations we will often work with $\mathcal{B}$ instead of $\hat{\mathcal{B}}$.  

Let $q' := (1 - \varepsilon|\hat{B}|)q$.
By Proposition~\ref{superprop}, for all $1 \leq i \leq \hat{k}$ we can remove $\varepsilon|\hat{B}| q$ vertices from each cluster $V_a$ belonging to $\hat{B}_i$ so that each edge $V_aV_b$ in $\hat{B}_i$ now corresponds to a $(2\varepsilon, d/2)$-superregular pair
 $(V_a, V_b)_{G'}$. Further, all clusters now have size $q'$ and for each edge $V_aV_b$ in $\hat{B}_i$ the pair $(V_a,V_b)_{G'}$ is a $2\varepsilon$-regular pair with density at least $d/2$ (by Fact \ref{slicinglemma}).
\COMMENT{AT: I deleted the footnote here as it is a standard property so you don't need to point out where it is being used.}
Add all the vertices we removed from the clusters to $V_0$ and observe that now, since $\varepsilon \ll \eta_{1}, 1/|\hat{B}|$, \COMMENT{AT: It is important that $\varepsilon \ll 1/|\hat{B}|$ in addition to $\varepsilon \ll \eta_{1}$. Do you see why?}
\begin{equation} \label{v0}
|V_0| \leq 3\eta_1 n.
\end{equation}
From now on, we will refer to the subclusters of size $q'$ as the clusters of $R$.\\

By considering a random partition of each cluster $V_a$, and applying  a Chernoff bound, one can obtain the following partition of each cluster.
\COMMENT{AT: rewrote} 
\begin{claim}\label{kodegseqclaim}
Let $V_a$ be a cluster. Then there exists a partition of $V_a$ into a red part $V_a^{red}$ and a blue part $V_a^{blue}$ such that $$||V_{a}^{red}| - |V_a^{blue}|| \leq \varepsilon q'$$ and $$||N_G(x) \cap V_a^{red}| - |N_G(x) \cap V_a^{blue}|| < \varepsilon q' \ \ \mbox{for all $x \in V(G)$}.$$
\end{claim}

Apply Claim~\ref{kodegseqclaim} to every cluster to yield a partition of $V(G) - V_0$ into red and blue vertices. In the next section, we will remove vertices of particular copies of $H$ in $G$ from their respective clusters and do so in such a way that we avoid all the red vertices of $G$.
After removing these vertices, we will be able to conclude that that each (modified) pair $(V_a, V_b)_{G'}$ is $(5\varepsilon, d/5)$-superregular\footnote{Where $V_aV_b$ is any edge in any $\hat{B}_i$ in $\hat{\mathcal{B}}$.} since $V_a^{red}$ and $V_b^{red}$ will have had no vertices removed from them. After the next section, we will only remove a bounded number of vertices from the clusters, which will not affect the superregularity of pairs of clusters in any significant way.

\subsection{Covering the exceptional vertices}\label{sectabsorb}
As in~\cite{kuhn},
given $x \in V_0$, we call a copy of $B \in \mathcal{B}$ \emph{useful for $x$} if there exist $r - 1$ clusters in $B$, each belonging to a different vertex class of $B$, such that $x$ has at least $\alpha q'$ neighbours in each cluster.
Denote by $k_x$ the number of copies of $B$ in $\mathcal{B}$ that are useful for $x$. The following calculation demonstrates that 
\begin{equation*}
k_x \beta q' \geq |V_0|.
\end{equation*} 
By (\ref{hier}) and (\ref{v0}), we have that 
\begin{align*}
 &  k_x|B|q' + (|\mathcal{B}| - k_x)(|B|q' - (1 - \alpha)q'\hat{c}(\omega + \sigma)) \\
 &  \geq d_G(x) - |V_0| \\
 &  \geq \left(1 - \frac{\omega + \sigma}{h} + \frac{\eta}{2}\right)q'|B||\mathcal{B}|,
\end{align*}
which implies $$(|\mathcal{B}| - k_x)(-(1 - \alpha)q'\hat{c}(\omega + \sigma)) \geq \left(-\frac{\omega + \sigma}{h} + \frac{\eta}{2}\right)q'h\hat{c}|\mathcal{B}|.$$
Rearranging, we get 
\begin{equation*}
k_x \geq \frac{|\mathcal{B}|\left(\frac{h\eta}{2} - \alpha (\omega + \sigma)\right)}{(\omega + \sigma)(1 - \alpha)}. 
\end{equation*}
Since $\alpha \ll \eta$, we have that
$$k_x \geq \frac{\eta |\mathcal{B}|}{4}.$$ 
Now as $|\mathcal{B}|q' \geq \frac{n}{2|B|}$ and $\eta_1 \ll \beta, \eta, 1/h$ we have that $$k_x \beta q' \geq \eta|\mathcal{B}|\beta q'/4 > 3 \eta_1 n \geq |V_0|.$$
Hence we can assign greedily each vertex $x \in V_0$ to a copy $B_x$ that is useful for $x$ and do so in such a way that at most $\beta q'$ vertices in $V_0$ are assigned to the same copy $B \in \mathcal{B}$.
Then for each copy $B_x \in \mathcal{B}$ that is useful for some $x \in V_0$ we can apply Lemma~\ref{keylem} to find a copy of $H$ containing $x$ which contains no red vertices. We do this as follows:

For each $x$, since $\varepsilon \ll \alpha$ and $x$ has at least $\alpha q'$ neighbours in $r - 1$ clusters belonging to different vertex classes of $B_x$, Claim~\ref{kodegseqclaim} implies that $x$ has at least $\alpha q'/4$ blue neighbours in each of these $r - 1$ clusters. 
Further, we can find $\alpha q'/4$ blue vertices in a cluster belonging to the vertex class of $B_x$ that does not necessarily contain any neighbours of $x$. Denote this vertex class of $B_x$ by $C_x$. 
Then it is easy to see that we can find subclusters $S_1, \ldots, S_r$ of $r$ clusters in $B_x$ such that: all vertices in $S_1 \cup \ldots \cup S_{r}$ are blue vertices; $|S_i| = \alpha q'/4$ for each $i$; every vertex in $S_1 \cup \ldots \cup S_{r-1}$ is a neighbour of $x$ in $G$. By Fact~\ref{slicinglemma}, each pair $(S_i, S_j)$, $1 \leq i < j \leq r$, corresponds to an $(8\varepsilon/\alpha)$-regular 
\COMMENT{AT: corrected the numbers in these sentences}
pair in $G'$ with density at least $d/3$. 
Using Lemma~\ref{keylem} with parameters $8\varepsilon/\alpha, d/3, \alpha q'/4, h - 1$, we find a copy of $H$ containing $x$. Since each $B \in \mathcal{B}$ has been assigned to at most $\beta q'$ vertices in $V_0$ and $\beta \ll \alpha$ (from (\ref{hier})), we may repeat the above argument to find copies of $H$ that contain each exceptional vertex in such a way that the copies are disjoint and contain no red vertices. Denote the $H$-tiling induced by these copies of $H$ by $\mathcal{H}_1$. 
Remove all the vertices lying in these copies of $H$ from their respective clusters. Observe that currently \[(1-\beta h)q' \leq |V_i| \leq q'\] for each $i$.\\

\subsection{Making the blow-up of each $B \in \mathcal{B}$ divisible by $h$}\label{sectdivisible}

For a subgraph $S \subseteq R$, let $V_G(S)$ denote the union of the clusters in $S$. We aim to apply Lemma~\ref{blowup} to each $\hat{B_i}$ in $\hat{\mathcal{B}}$ to find an $H$-tiling that covers every vertex of $V_G(\hat{B_i})$.
Combining these $H$-tilings with $\mathcal{H}_1$ will result in a perfect $H$-tiling in $G$ as desired. Recall that, for each $1 \leq i \leq \hat{k}$, $\hat{G}_i$ is the $r$-partite subgraph of $G'$ whose $j$th vertex class is the union of all those clusters contained in the $j$th vertex class of $\hat{B}_i$, for each $1 \leq j \leq r$. Further, recall that $G^{*}_i$ is the complete $r$-partite graph on the same vertex set as $\hat{G}_i$. 
To apply Lemma~\ref{blowup} to each $\hat{B}_i$ in $\hat{\mathcal{B}}$ we require that each $G_i^*$ contains a perfect $H$-tiling.
To guarantee the existence of these perfect $H$-tilings we will apply Lemma~\ref{generalkolem}. To use Lemma~\ref{generalkolem} on $G_i^*$ we require that $|V(\hat{G_i^*})|$ is divisible by $h$.
When we first chose our $\hat{B}$-tiling this was the case. Indeed, as each $\hat{B_i}$ contained a perfect $H$-tiling and every cluster $V_i$ was the same size, $|V(G_i^*)|$ was divisible by $h$. However, in the last section we took out vertices from $G$ in a greedy way, changing the sizes of the clusters in $R$. Hence we cannot guarantee that $|V(G_i^*)|$ is still divisible by $h$ for each $i$. 
Now we will take out a further bounded number of copies of $H$ in $G$ to ensure $|V(G_i^*)|$ is divisible by $h$ for each $1 \leq i \leq \hat{k}$. In fact, we will ensure $|V_G(B)|$ is divisible by $h$ for each $B \in \mathcal{B}$.

We now split into two cases: when $r \geq 3$ and when $r = 2$. When $r \geq 3$ we have that $\textnormal{hcf}_{\chi}(H) = 1$ and this property will be central to our argument. For $r = 2$, we have that $\textnormal{hcf}_{c}(H) = 1$ and $\textnormal{hcf}_{\chi}(H) \leq 2$. The former property will provide us an easy way of removing copies of $H$ from $V(G)$ to ensure $|V_G(B)|$ is divisible by $h$ for each $B \in \mathcal{B}$. Further, we will not need to use the property that $\textnormal{hcf}_{\chi}(H) \leq 2$ in our argument. The only time we (implicitly) use the property that $\textnormal{hcf}_{\chi}(H) \leq 2$ will be when we apply Lemma~\ref{generalkolem}.

\subsubsection{Case 1: $r \geq 3$}\label{sectcase1}

To assist in our argument, we define an auxiliary graph $F$ whose vertices are the copies of $B$ in $\mathcal{B}$ and for $B_1, B_2 \in V(F)$, we let $B_1B_2$ be an edge in $F$ if and only if there exists a vertex $x$ in $V_R(B_1)$ and $r - 1$ vertices in $V_R(B_2)$ (or vice versa) such that these $r$ vertices induce a $K_r$ in $R$. Assume $F$ is connected and let $B_1B_2$ be an edge in $F$. Then we may apply Lemma~\ref{keylem} to find $h - 1$ disjoint copies of $H$ which each have one vertex in $V_G(B_1)$ and all other vertices in $V_G(B_2)$ (or vice versa). 
This means that we can remove at most $h - 1$ copies of $H$ to ensure $V_G(B_1)$ is divisble by $h$. Continuing in this way we can `shift the remainders mod $h$' along a spanning tree of $F$ to ensure $|V_G(B)|$ is divisible by $h$ for each $B \in \mathcal{B}$. (Indeed, since $n$ is divisible by $h$ we have that $\sum_{B \in \mathcal{B}}|V_G(B)|$ is divisible by $h$.) 

So assume $F$ is not connected. Let $\mathcal{C}$ be the set of all components of $F$. For $C \in \mathcal{C}$ we will write $V_R(C)$ for the set of clusters in $R$ belonging to copies of $B$ in $C$ and $V_G(C)$ for the union of said clusters. 
In what follows our aim is to remove a bounded number of copies of $H$ to ensure that for each component $C \in \mathcal{C}$ we have that $|V_G(C)|$ is divisible by $h$. Then we can apply our previous argument to spanning trees of each component to achieve that $|V_G(B)|$ is divisible by $h$ for each $B \in \mathcal{B}$.

% Observe that the reduced graph $R$ by this point in the proof would have degree sequence $d_{R,1} \leq d_{R,2} \leq \ldots \leq d_{R,|R|}$ such that
% \begin{equation} \label{rdegseq} d_{R,i} \geq \frac{h-\omega-\sigma}{h}|R| + \frac{\sigma}{\omega}i + \frac{\eta |R|}{4} \ \ \mbox{for all} \ \ 1 \leq i \leq \frac{\omega |R|}{h}.\end{equation}

Call vertices in $R$ of degree at least 
\begin{equation} \label{bigdeg} 
(1 - \omega/h + \eta/4)k 
\end{equation}
\emph{big}. If a vertex is not big, call it \emph{small}. Note by (\ref{rdegseqnew}) that all but at most $\omega k/h - 1$ vertices in $R$ are big.
\begin{claim} \label{degclaim}
Let $C_1, C_2 \in \mathcal{C}$, $C_1 \neq C_2$ and let $a \in V_R(C_2)$. Then
$$|N_R(a) \cap V_R(C_1)| < \left ( 1 - \frac{\omega + \sigma}{h} + \frac{\eta}{4} \right )|V_R(C_1)|.$$
\end{claim}

\begin{proof} Recall that $B$ has width $\omega\hat{c}$. Suppose Claim~\ref{degclaim} is false. Then there exists some $B_0 \in \mathcal{B}$ such that $B_0 \in C_1$ and
$$|N_R(a) \cap B_0| \geq \left(1 - \frac{\omega + \sigma}{h} + \frac{\eta}{4} \right)|B_0| = (r-2)\omega\hat{c}  + \frac{\eta h\hat{c}}{4}.$$
Thus $a$ must have neighbours in at least $r - 1$ vertex classes of $B_0$. We can therefore construct a copy of $K_r$ in $R$ which consists of $a$ together with $r - 1$ of its neighbours in $B_0$. 
But by definition of the auxiliary graph $F$, we must have that $B_0$ is adjacent in $F$ to the copy of $B$ in $\mathcal{B}$ that contains $a$. 
This contradicts that $C_1$ and $C_2$ were different components of $F$. Thus Claim~\ref{degclaim} holds.
\end{proof}

\begin{claim} \label{edgeclaim}
There exist components $C_1, C_2 \in \mathcal{C}$, $C_1 \neq C_2$, a big vertex $x_1 \in V(R)$ and another (not necessarily big) vertex $x_2 \in V(R)$  such that $x_1 \in V(C_1)$, $x_2 \in V(C_2)$ and $x_1x_2 \in E(R)$. 
\end{claim}

\begin{proof} Take some big vertex $x \in V(R)$. 
Then $x$ is in $V_R(C_x)$ for some component $C_x$ of $F$. 
Assume $|C_x| \geq (1 - \omega/h + \eta/4)k$, as otherwise $x$ has a neighbour in $R$ outside of $C_x$ and we are done. Since $r \geq 3$, 
$$|R \setminus V_R(C_x)| \leq (\omega/h - \eta/4)k < (1 - \omega/h + \eta/4)k.$$

If $R \setminus V_R(C_x)$ contains any big vertex $y$, then $y$ has a neighbour in $V_R(C_x)$ since $|R \setminus V_R(C_x)| < (1 - \omega/h + \eta/4)k$ and we are done. Hence assume all big vertices are in $V_R(C_x)$. 
Then all vertices in $R \setminus V_R(C_x)$ are small vertices.
Let $z$ be a small vertex in $R \setminus V_R(C_x)$. Since $r \geq 3$, $$d_R(z) \geq (1 - (\omega + \sigma)/h + \eta/4)k \geq (\omega/h + \eta/4)k.$$ 
Since there are at most $\omega k/h - 1$ small vertices in $R$, we have that $z$ has a neighbour $w$ which is a big vertex. But then $w \in V_R(C_x)$. Thus Claim~\ref{edgeclaim} holds.
\end{proof}

\begin{claim} \label{krclaim}
There exists a copy $K'$ of $K_r$ in $R$ which has vertices in at least two components of $F$.
\end{claim}

\begin{proof} By Claim~\ref{edgeclaim}, there exist components $C_1, C_2 \in \mathcal{C}$, a big vertex $x_1 \in V(R)$ and another vertex $x_2 \in V(R)$ such that $x_1 \in V_R(C_1)$, $x_2 \in V_R(C_2)$ and $x_1x_2 \in E(R)$. By (\ref{rdegseqnew}) and (\ref{bigdeg}), $x_1$ and $x_2$ have a common neighbourhood of size at least 
$$((r-3)\omega/h + \eta/2)k.$$
If $r = 3$, then we choose $x_3$ in the common neighbourhood of $x_1$ and $x_2$, and we are done.
So assume $r \geq 4$. Since there are at most $\omega k/h$ small vertices, we can choose a big vertex $x_3$ in the common neighbourhood of $x_1$ and $x_2$. 
Then $x_1, x_2$ and $x_3$ have a common neighbourhood of size at least 
$$((r-4)\omega/h + 3\eta/4)k.$$
If $r = 4$, then we choose $x_4$ in the common neighbourhood of $x_1, x_2$ and $x_3$ and we are done. Otherwise $r \geq 5$ and we continue as before. Thus Claim~\ref{krclaim} holds.
\end{proof}

For such a copy $K'$ of $K_r$ in $R$, we now show that we can take out a bounded number of copies of $H$ from the clusters corresponding to the vertices of $K'$ in such a way that that leaves one of the components $C \in \mathcal{C}$ with $|V_G(C)|$ divisible by $h$. We use Theorem~\ref{partitionthm} and Lemma~\ref{keylem} to achieve this.
We will then repeat this process to ensure $|V_G(B)|$ is divisible by $h$ for each $B \in \mathcal B$.
\COMMENT{AT: be careful how you explain things! I've added this sentence as by itself, what you wrote will have confused the reader. That is, the sentences you wrote didn't precisely match up with the statement of the claim.}

\begin{claim}
There exists $t \in \mathbb{N}$ such that by removing at most $t + (|\mathcal{B}| - |\mathcal{C}|)(h - 1)$ copies of $H$ from $G$ we can ensure $|V_G(B)|$ is divisible by $h$ for each $B \in \mathcal{B}$. 
\end{claim}

\begin{proof}
Firstly, for each component $C \in \mathcal{C}$ we will remove copies of $H$ to ensure $|V_G(C)|$ is divisible by $h$. Apply Claim~\ref{krclaim} to find a copy $K'$ of $K_r$ in $R$ which has vertices in at least two components of $F$.
Let $C^{*}$ be a component of $F$ which contains at least one vertex of $K'$. Let $p$ be the number of vertices of $K'$ contained in $C^{*}$ and observe that $1 \leq p \leq r-1$.
Let $0 \leq g \leq h-1$ such that $|V_G(C^{*})| \equiv g\mod h$. If $g = 0$ then $|V_G(C^{*})|$ is divisible by $h$ and we consider the graphs $F_1 := F - V(C^{*})$ and $R_1 := R - V_R(C^{*})$. So assume $1 \leq g \leq h-1$.
Observe that we can apply Lemma~\ref{keylem} to find any bounded number of disjoint copies of $H$ in $G$ in the clusters of $K'$ (see the end of Section~\ref{sectabsorb}).
For any copy $H'$ of $H$ in $G$ in the clusters of $K'$ there are precisely $p$ colour classes of some colouring $c$ of $H'$ contained in the clusters of $K'$ in $V_G(C^{*})$. Moreover, given any colouring $c$ of $H$ and $p$-subset $P$ contained in $D_c$ we can find any bounded number of disjoint copies $H'$ of $H$ in $G$ with colouring $c$ in the clusters of $K'$ so that the colour classes of $H'$ in $V_G(C^*)$ correspond to the $p$-subset $P$. Moreover there exists $j \in \{1, \ldots, z_p\}$ such that $P = A_{p,c,j}$
(recall this notation from Definition~\ref{defy}). Thus removing such a copy $H'$ of $H$ from $G$ would result in removing precisely $S_{p,c,j}$ vertices from $V_G(C^*)$. 
By Theorem~\ref{partitionthm}, there exist a collection of non-negative integers $\{a_{p,c,i}: c \in C_H, 1 \leq i \leq z_p\}$ and $\bar{a} \in \mathbb{N}$ such that
$$a_{p,c,i} \leq \bar{a} \ \ \mbox{for all} \ c \in C_H, 1 \leq i \leq z_p,$$ and
$$g \cdot \sum_{c \in C_H}\sum_{i = 1}^{z_p}a_{p,c,i}S_{p,c,i} \equiv g\mod h.$$ 
Hence we can remove
$$g \cdot \sum_{c \in C_H}\sum_{i = 1}^{z_p}a_{p,c,i} 
\leq (h-1)\bar{a}|C_H|z_p$$ 
suitable disjoint copies of $H$ in $G$ in the clusters of $K'$ to make $|V_G(C^{*})|$ divisible by $h$.\COMMENT{AT: again added details to help the reader some more}

Next we consider graphs $F_1 := F - V(C^{*})$ and $R_1 := R - V_R(C^{*})$.
Let $k_1 := |R_1|$. Claim~\ref{degclaim} and (\ref{rdegseqnew}) together give us that $R_1$ has degree sequence $d_{R_1, 1} \leq \ldots \leq d_{R_1, k_1}$ where
 \begin{equation*}\label{r1degseqnew}
 d_{R_1,i} \geq \left(1 - \frac{\omega+\sigma}{h}\right)k_1 + \frac{\sigma}{\omega}i + \frac{\eta k_1}{4} \ \ \mbox{for all $1 \leq i \leq \frac{\omega k_1}{h}$.}
\end{equation*}

Suppose $|\mathcal{C}| \geq 3$. Arguing as in Claims~\ref{edgeclaim} and~\ref{krclaim} we can find a copy $K_1'$ of $K_r$ in $R_1$ which has vertices in at least two components of $F_1$.
\COMMENT{AT: again made things more precise here. We cannot apply Claim~\ref{krclaim} to find a copy $K_1'$ of $K_r$ in $R_1$ as it is not a claim about $R_1$. However, (\ref{r1degseqnew}) does allow us to argue analogously to Claims 7.3--7.4}
 Let $C^{**}$ be a component of $F$ which contains at least one vertex of $K_1'$. 
As before by removing at most $(h-1)\bar{a}|C_H|z_p$ copies of $H$ from the clusters of $K'_1$ we can make $|V_G(C^{*})|$ divisible by $h$.
Since $|G|$ is divisible by $h$, we can continue in this way to make $|V_G(C)|$ divisible by $h$ for each component $C \in \mathcal{C}$. We then apply the `shifting the remainders mod $h$' argument mentioned earlier during the `$F$ connected' case to guarantee that $|B|$ is divisible by $h$ for each $B \in \mathcal{B}$.
In this process we removed at most $(|\mathcal{C}| - 1)(h-1)\bar{a}|C_H|z_p$ disjoint copies of $H$ from $G$. 
Each time we use the `shifting the remainders mod $h$' argument on a connected component $C \in \mathcal{C}$ we remove at most $(|C| - 1)(h - 1)$ disjoint copies of $H$ in $G$. 
Hence overall we remove at most $(|\mathcal{C}| - 1)(h-1)\bar{a}|C_H|z_p + (|\mathcal{B}| - |\mathcal{C}|)(h - 1)$ disjoint copies of $H$ in $G$. Denote this $H$-tiling (formed from these copies of $H$) by $\mathcal{H}_2$.
\end{proof}

Observe that now $$(1 - 2h\beta)q' \leq |V_i| \leq q'$$ for each $i$ since we only removed a bounded number of vertices from $G$.

\subsubsection{Case 2: $r = 2$}\label{sectcase2}
As in the statement of Theorem~\ref{partitionthm}, let $b$ be the number of components of $H$ and $t_1, \ldots, t_b$ be the sizes of the components of $H$. By Theorem~\ref{partitionthm}, there exists a collection of non-negative integers $\{a_{i}:  1 \leq i \leq b\}$ and $\bar{a} \in \mathbb{N}$ such that $$a_i \leq \bar{a} \ \ \mbox{for all} \ \ 1 \leq i \leq b,$$  and $$\sum_{i = 1}^{b}a_it_i \equiv 1\mod h.$$
Let $B_1, B_2 \in \mathcal{B}$.
If $|V_G(B_1)| \equiv 0\mod h$, define $\mathcal{B}_1 := \mathcal{B}\setminus B_1$. If not, let $p \in \{1, \ldots, h - 1\}$ such that $|V_G(B_1)| \equiv p\mod h$. 
Remove $p \sum_{i = 1}^{b}a_i$ copies of $H$ from $V_G(B_1) \cup V_G(B_2)$ in the following way: For each $1 \leq i \leq b$, remove $pa_i$ copies of $H$ from $V_G(B_1) \cup V_G(B_2)$ such that the component of order $t_i$ is in $V_G(B_1)$ and all other components are in $V_G(B_2)$.\footnote{We use Lemma~\ref{keylem} to do this.} 
Since $p \sum_{i = 1}^{b}a_it_i \equiv p\mod h$, by removing these $p \sum_{i = 1}^{b}a_i$ copies of $H$ from $V_G(B_1) \cup V_G(B_2)$ we now have that $|V_G(B_1)|$ is divisible by $h$. Define $\mathcal{B}_1 := \mathcal{B}\setminus B_1$. 

Let $B_1', B_2' \in \mathcal{B}_1$. If $|V_G(B_1')| \equiv 0\mod h$, define $\mathcal{B}_2 := \mathcal{B}_1\setminus B_1'$. If not, let $p' \in \{1, \ldots, h-1\}$ such that $|V_G(B_1')| \equiv p'\mod h$. Remove $p'\sum_{i = 1}^{b}a_i$ copies of $H$ from $V_G(B_1') \cup V_G(B_2')$ in the same way as before. Define $\mathcal{B}_2 := \mathcal{B}_1\setminus B_1'$. Continuing in the same way, we see that by removing at most 

\begin{equation}\label{requals2eqn}
    (|\mathcal{B}| - 1)(h - 1)b\bar{a}
\end{equation} 
copies of $H$ we can ensure that $|B|$ is divisible by $h$ for each $B \in \mathcal{B}$.  % \footnote{Since $\sum_{B \in \mathcal{B}}|V_G(B)|$ is divisible by $h$,
% after we remove copies of $H$ from $\mathcal{B}_{|\mathcal{B}| - 2} = \{B_1^{(|\mathcal{B}| - 2)}$, $B_2^{(|\mathcal{B}| - 2)}\}$ (if necessary), both
% $|V_G(B_1^{(|\mathcal{B}| - 2)})|$ and $|V_G(B_2^{(|\mathcal{B}| - 2)})|$ will be divisible by $h$. This explains the presence of the $(|\mathcal{B}| - 1)$ term in (\ref{requals2eqn}).} 
Denote this $H$-tiling (formed from these copies of $H$) by $\mathcal{H}_2$.

Observe that now $$(1 - 2h\beta)q' \leq |V_i| \leq q'$$ for each $i$ since we only removed a bounded number of vertices.

\subsection{Completing the perfect tiling}\label{sectblowup}

As we related at the beginning of Section~\ref{sectdivisible}, we aim to apply Lemma~\ref{blowup} to each $\hat{B}_{i} \subseteq R$ ($1 \leq i \leq \hat{k}$) where the vertices of $R$ are the now modified clusters -- modified by the removing of copies of $H$ in previous sections.
Recall that, for each $1 \leq i \leq \hat{k}$, $\hat{G}_i$ is the $r$-partite subgraph of $G'$ whose $j$th vertex class is the union of all those clusters contained in the $j$th vertex class of $\hat{B}_i$, for each $1 \leq j \leq r$. 
Observe that in Section~\ref{sectdivisible} we made $|\hat{G}_i| = |V_G(\hat{B}_i)|$ divisible by $h$ for each $i$.
Further, $$ (1 - 2h\beta)q' \leq |V_i| \leq q'$$ for each $i$.
Recall that $G^{*}_i$ is the complete $r$-partite graph on the same vertex set as $\hat{G}_i$. Since $0 < 2h\beta \ll \sigma/\omega, 1 - \sigma/\omega, 1/h$ by (\ref{hier}), we can apply Lemma~\ref{generalkolem} to conclude that each $G^{*}_i$ contains a perfect $H$-tiling. 

Furthermore, pairs of clusters that correspond to edges of $\hat{B}_i$ are still $(6\varepsilon, d/6)$-superregular.
Indeed, in Section~\ref{sectabsorb} we removed copies of $H$ which avoided red vertices, resulting in each pair of clusters (in a copy of $H$) being $(5\varepsilon, d/5)$-superregular. Then, in Section~\ref{sectcase1}, or Section~\ref{sectcase2} if $r = 2$, we removed only a constant number of vertices from each cluster. Hence each pair of clusters (in a copy of $H$) is $(6\varepsilon, d/6)$-superregular. 

We now have all we need to apply Lemma~\ref{blowup} to find a perfect $H$-tiling $\hat{\mathcal{H}}_i$ in $\hat{G}_i$ for each $1 \leq i \leq \hat{k}$.
Then $$\mathcal{H}_1 \cup \mathcal{H}_2 \cup \hat{\mathcal{H}}_1 \cup \ldots \cup \hat{\mathcal{H}}_{\hat{k}}$$ is a perfect $H$-tiling in $G$. Hence we have proved Theorem~\ref{mainthmsigma}.

\begin{section}{Acknowledgements}
The first author would like to thank Pat Devlin for a helpful conversation at Building Bridges II.
\end{section}

\end{document}